%% file: eTRN16a.tex
\documentclass[reqno,11pt]{amsart}
\include{environ}

\include{combined_sym}

\begin{document}

\title[Pseudo-time regularization methods]
      {Pseudo-time regularization for PDE with solution-dependent diffusion}

\author[S. Pollock]{Sara Pollock}
\email{sara.pollock@wright.edu}

\address{Department of Mathematics and Statistics\\
         Wright State University\\ 
         Dayton, OH 45435}


\date{\today}



\input{abs}

\maketitle




\input{body}
\section{Acknowledgments}
   \label{sec:ack}
The author would like to thank William Rundell and Yunrong Zhu for numerous
interesting discussions on the topics addressed here, and for input on a draft
of this manuscript. 

\bibliographystyle{abbrv}

\bibliography{refsTRN}



\end{document}

%% file: environ.tex
\usepackage{color} 
\usepackage{ifpdf}
\ifpdf
    \usepackage[pdftex]{graphicx}
    \usepackage[pdftex]{hyperref}
    \hypersetup{
        unicode=false,          
        pdftoolbar=true,        
        pdfmenubar=true,        
        pdffitwindow=false,     
        pdfstartview={FitH},    
        pdftitle={SP Article},      
        pdfauthor={Sara Pollock},   
        pdfsubject={Mathematics},    
        pdfcreator={Sara Pollock},  
        pdfproducer={Sara Pollock}, 
        pdfkeywords={PDE, analysis, numerical analysis}, 
        pdfnewwindow=true,      
        colorlinks=true,        
        linkcolor=red,          
        citecolor=blue,         
        filecolor=magenta,      
        urlcolor=cyan           
    }

\else
    \usepackage{graphicx}
    \usepackage{pstricks}
    
    \newcommand{\href}[2]{#2}
\fi

\usepackage{times}
\usepackage{amsfonts}
\usepackage[font={small}]{caption}

\usepackage{amsmath}
\usepackage{amsthm}
\usepackage{amssymb}
\usepackage{amsbsy}
\usepackage{amscd}
\usepackage{extarrows}

\usepackage{enumerate}
\usepackage{verbatim}
\usepackage{subfigure}

\newtheorem{theorem}{Theorem}[section]
\newtheorem{corollary}[theorem]{Corollary}
\newtheorem{lemma}[theorem]{Lemma}

\newtheorem{assumption}[theorem]{Assumption}

\newtheorem{example}[theorem]{Example}
\newtheorem{remark}[theorem]{Remark}

\newtheorem{algorithm}[theorem]{Algorithm}
\newtheorem{criteria}[theorem]{Criteria}
\newtheorem{condition}[theorem]{Condition}

\numberwithin{equation}{section}  

  \newcounter{mnote}
  \let\oldmarginpar\marginpar
    \renewcommand\marginpar[1]{\-\oldmarginpar[\raggedleft\footnotesize #1]%
    {\raggedright\footnotesize #1}}

\newenvironment{enumerateY}
{\begin{list}{{\it(\roman{enumii})}}{
\usecounter{enumii}
\leftmargin 2.5em\topsep 0.5em\itemsep -0.0em\labelwidth 50.0em}}
{\end{list}}

\newenvironment{enumerateX}
{\begin{list}{\arabic{enumi})}{
\usecounter{enumi}
\leftmargin 2.5em\topsep 0.5em\itemsep -0.0em\labelwidth 50.0em}}
{\end{list}}

\definecolor{myblue}{rgb}{0.2,0.2,0.7}
\definecolor{mygreen}{rgb}{0,0.6,0}
\definecolor{mycyan}{rgb}{0,0.6,0.6}
\definecolor{myred}{rgb}{0.9,0.2,0.2}
\definecolor{mymagenta}{rgb}{0.9,0.2,0.9}
\definecolor{mywhite}{rgb}{1.0,1.0,1.0}
\definecolor{myblack}{rgb}{0.0,0.0,0.0}



%% file: combined_sym.tex
\renewcommand{\div}{{\operatorname{div}}}
\newcommand{\eps}{\varepsilon}

\newcommand{\GM}{\gamma_{\text{MAX}}}
\newcommand{\gmo}{\gamma_{\text{MONO}}}

\newcommand{\R}{{\mathbb R}}       

\newcommand{\cA}{{\mathcal A}}

\newcommand{\cE}{{\mathcal E}}
\newcommand{\cF}{{\mathcal F}}

\newcommand{\cL}{{\mathcal L}}

\newcommand{\cP}{{\mathcal P}}

\newcommand{\cT}{{\mathcal T}}
\newcommand{\cU}{{\mathcal U}}
\newcommand{\cV}{{\mathcal V}}

\newcommand{\w}{{\tt w}}  
\newcommand{\fl}{{\cF^e}} 
\newcommand{\lin}{{\cL^e}} 
\DeclareMathAlphabet{\mathpzc}{OT1}{pzc}{m}{it}

\newcommand{\f}{\frac}

\newcommand{\tol}{{\tt tol}}  
\newcommand{\itmax}{{\tt itmax}}  
\DeclareMathOperator\sign{sign}
\DeclareMathOperator\meas{meas}
\DeclareMathOperator\ceil{ceil}

\newcommand{\on}{\text{ on }}
\newcommand{\oor}{\text{ or }}

\newcommand{\an}{\text{ and }}

\newcommand{\inn}{\text{ in }}

\newcommand{\tforall}{\text{ for all }}

\newcommand{\bigo}{{\mathcal O}} 

\newcommand{\rest}{\big|}

\newcommand{\argmin}{\text{argmin\,}}
\newcommand{\argmax}{\text{argmax\,}}

\newcommand{\essup}{\text{ess sup\,}}

\newcommand{\grad}{\nabla} 

\newcommand{\goto}{\rightarrow}

\newcommand{\norm}[1]{\ensuremath{\lVert{#1} \rVert}}
\newcommand{\nr}[1]{\norm{#1}} 

\newcommand{\pa}{\partial}

\definecolor{blue}{rgb}{0.2,0.2,0.7}
\definecolor{red}{rgb}{0.7,0.3,0.1}
\definecolor{cyan}{rgb}{0.2,0.5,0.6}
\usepackage{mathtools}
\usepackage{stmaryrd}

%% file: abs.tex
\begin{abstract}
This work unifies pseudo-time and inexact regularization 
techniques for nonmonotone classes of partial differential equations,
into a regularized pseudo-time framework.
Convergence of the residual at the predicted rate is investigated through 
the idea of controlling the linearization error, and regularization parameters 
are defined following this analysis, then assembled in an adaptive algorithm. 
The main innovations of this paper include the introduction of
a Picard-like regularization term scaled by its cancellation effect
on the linearization error to stabilize the Newton-like iteration;
an updated analysis of the regularization parameters in terms of 
minimizing an appropriate quantity; and, strategies to accelerate
the algorithm into the asymptotic regime. Numerical experiments demonstrate 
the method on an anisotropic diffusion problem where the Jacobian
is not continuously differentiable, and a model problem with steep gradients
and a thin diffusion layer.
\end{abstract}

%% file: body.tex
\section{Introduction}
This paper is concerned with the finite element approximation to 
second order quasilinear elliptic equations in divergence form, 
\begin{align}\label{eqn:PDE}
-\div(\kappa(u) \grad u) &= f, \inn \Omega \\
u & = 0, \on \pa \Omega, \label{eqn:paOmega}
\end{align}
for polygonal doman $\Omega \subset \R^2$. 
Nonlinear diffusion problems are ubiquitous throughout science and
engineering applications, appearing in applications such as heat conduction,
groundwater flow, diffusion of contaminants and flow in porous media
\cite{CaRa94,Crank75,FKCKM03,HlKrMa94}.
Here, the nonlinear diffusion coefficient $\kappa(u)$ may be thought of as a 
scalar quantity, or in the more general anisotropic case as a matrix coefficient
with entries $\kappa_{ij}(u), ~j = 1,2$,
where \eqref{eqn:PDE} has the expansion
\begin{align}\label{eqn:anisotropic}
-\sum_{i,j = 1}^n \f{\pa}{\pa x_j} \left( \kappa_{ij}(u) \f{\pa }{\pa x_j}u\right) 
           = f, \inn \Omega,
\end{align}
with the ellipticity condition: For some $\eta > 0$ 
\begin{align}\label{eqn:elliptic}
\sum_{i,j=1}^2 \kappa_{ij}(s)\xi_i \xi_j \ge \eta\sum_{i = 1}^2 \xi_i^2,
\end{align}
for any $s \in \R$,  and all $\xi = (\xi_1,\xi_2) \in \R^2$.

As remarked in~\cite{HlKrMa94}, while \eqref{eqn:PDE} for scalar-valued 
$\kappa(u)$ may be solved by the Kirchhoff transform (see, {\em e.g.,} 
\cite{Crank75}),
this technique does not carry over to the anisotropic case, or to lower order
solution-dependent terms.  
The nonlinear diffusion problem \eqref{eqn:PDE} is generally in the class of 
nonmonotone problems; that is 
$\int_{\Omega}\left\{ \left( \kappa(v) \grad v-\kappa(w) \grad w\right) \cdot \grad(v - w)
\right\} > 0$,
is not guaranteed to hold for each $u, v$ is the solution space, {\em e.g.,} 
$u,v \in H_0^1(\Omega)$.
While convergence and optimality of finite element methods for monotone classes of
quasilinear problems have been recently investigated in \cite{BDK12,GaMoZu11}, and the 
references therein, nonmonotone classes of problems are less understood.
In particular, convergence results rely on sufficiently fine global mesh
conditions, and sufficiently close initial guesses to assume convergence of 
Newton-iterations to solve the discrete nonlinear problem
\cite{BiGi13,DoDu75,GuPa07,HTZ09a}. 
Other recent work~\cite{ErnVor13} includes this problem class in an adaptive framework
of incomplete linear and nonlinear solves, but with the implicit assumption that
the sequence of solution iterates is convergent to the solution. 
As described in previous work by the author in \cite{Pollock14a,Pollock15a,Pollock15b},
the Newton iterations cannot realistically be assumed to converge, and are often
observed to diverge,
especially for problems of the form \eqref{eqn:PDE} which may contain steep gradients
and thin internal layers in the solution-dependent diffusion coefficient $\kappa(u)$.
The interest in the 
current investigation is to develop a regularized adaptive method and understand 
the residual convergence of the discrete nonlinear problem without these assumptions.
It is of particular interest to allow a solution process to start on a coarse mesh which 
is refined adaptively, to uncover an efficient and accurate discretization.  
Future work will directly address convergence of the discrete solution to the 
weak solution of \eqref{eqn:PDE}.

The weak form of \eqref{eqn:PDE}-\eqref{eqn:paOmega} is given by:
Find $u \in \cU$ such that
\begin{align}\label{eqn:weak}
B(u;u,v) = \int_\Omega fv, ~\tforall v \in \cV,
\end{align}
for solution space $\cU$ and test space $\cV$, with
\begin{align}\label{eqn:defB}
B(u; u, v) = \int_\Omega \kappa(u) \grad u \cdot \grad v.
\end{align}
Based on the analysis of \cite{DoDuSe71,HlKrMa94,Zhang2000}, the following set of 
conditions in addition to the uniform ellipticity~\eqref{eqn:elliptic} is
sufficient to assure existence and uniqueness
of the weak solution $u \in H_0^1(\Omega)$ of ~\eqref{eqn:weak}.
\begin{assumption}\label{assume:pdeWP} Assume the data satisfy the following 
boundedness and Lipschitz conditions.
\begin{enumerate}
\item Boundedness of the diffusion coefficient
\begin{align}\label{eqn:kappaBounded} 
\essup_{s \in \R }\kappa_{ij}(s) \le C_\kappa, ~i,j = 1,2.
\end{align}
\item Boundedness of the source: $f \in L_2(\Omega)$ satisfies 
\begin{align}\label{eqn:fBounded} 
\essup_{x \in \Omega }f(x) \le C_f.
\end{align} 
\item Lipschitz continuity of the diffusion coefficient
\begin{align}\label{eqn:kappaLip}
|\kappa_{ij}(s) - \kappa_{ij}(t)| \le \omega_\kappa |s-t|, ~j = 1,2, 
~\tforall s, t \in \R.
\end{align}
\end{enumerate}
For scalar-valued $\kappa$, the above should be interpreted with 
$\kappa_{11} = \kappa_{22} = \kappa$, and $\kappa_{12} = \kappa_{21}= 0$.
\end{assumption}

A finite dimensional, or discrete problem corresponding to ~\eqref{eqn:defB}, 
is given by: Find $u \in \cU_k$ such that 
\begin{align}\label{eqn:abs_discrete}
B(u;u,v) = \int_\Omega fv, ~\tforall v \in \cV_k \subset \cV,
\end{align}
where $\cU_k \subset \cU$ and $\cV_k \subset \cV$ are finite dimensional subspaces
of the solution and test spaces.  For the remainder of this paper, $\cU$ and
$\cV$ are assumed subsets of $H_0^1(\Omega)$. The trial and test spaces 
$\cU_k = \cV_k$, now referred to as $\cV_k$, are taken as the 
$C^0$ continuous $\cP^p$ finite element
spaces of polynomials of degree $p$ over each mesh element,
corresponding to nested mesh partitions, $\cT_k$, which are conforming in the 
sense of ~\cite{Ciarlet}.

In the current discussion, a Newton-like method is applied so solve the discrete
nonlinear equations induced by \eqref{eqn:abs_discrete}.  The advantages of 
a Newton-like approach include fast convergence in the asymptotic regime, and the
mesh-independence principle, as in for instance ~\cite{DePo92}.   
To understand the convergence of sequence of linear equations used to approximate
the solution of the discrete nonlinear problem, we require some control of the
Jacobian.  For this reason, the following conditions on the problem data are 
considered, in addition to those for well-posedness of the PDE, 
namely, the ellipticity condition \eqref{eqn:elliptic} and 
Assumption \ref{assume:pdeWP}.
\begin{assumption}[Assumptions on the problem data]\label{assume:pdeN}
The following assumptions are made on the diffusion coefficient 
$\kappa'(u)$, componentwise, $\kappa_{ij}'(u), ~ j = 1,2$.
\begin{enumerate} 
\item $\kappa'$ is bounded. In particular
\begin{align}\label{eqn:Kprime_bounded}
\nr{\kappa'(s)\xi} \le C_B \nr{\xi} ~\tforall s \in \R ~\an~ \xi \in \R^2.
\end{align}   
\item $\kappa'$ satisfies the Lipschitz condition
\begin{align}\label{eqn:LipKprime}
|\kappa_{ij}'(s) - \kappa_{ij}'(t)| \le \omega_B|s-t|, ~j=1,2, 
~\tforall s,t \in \R.
\end{align} 
\end{enumerate}
\end{assumption}

\begin{remark}[Standard problem classes]\label{remark:known_classes} 
The two existing convergence results for adaptive methods for nonmonotone problems
of the form \eqref{eqn:PDE} are developed in ~\cite{HTZ09a} using continuous
Galerkin, in particular linear elements; and, \cite{BiGi13} using discontinuous
Galerkin methods.

The results of \cite{HTZ09a} are adapted from the analysis of uniform methods
in \cite{CaRa94}.  Their approach assumes $W^{1,p}$, $p > 2$ regularity of the 
solution $u$, and bounded $\kappa(s), \kappa'(s)$ and $\kappa''(s)$ for all $s \in \R$.
The results of ~\cite{BiGi13} are adapted from the analysis of uniform
methods in \cite{GuPa07}, based on the results of \cite{DoDu75}.  
For these results, it is assumed that the solution 
$u \in H^2(\Omega) \cap W^{1,\infty }(\Omega)$, and $\kappa(s)$ is twice 
continuously differentiable with bounded $\kappa(s), \kappa'(s)$ and $\kappa''(s)$ 
for all $s \in \R$.

Both frameworks explicitly assume a uniformly small initial meshsize 
and either implicitly or explicitly  assume the convergence of the iterative method 
used to solve the discrete nonlinear system on each refinement. 
\end{remark}

The following notation is used throughout the paper. Where not 
otherwise specified, the norm $\nr{\cdot}$ denotes the $L_2$ norm. The integral
$\int_{\Omega} u v$ is sometimes denoted $( u\, ,  v)$; and,
$\langle z, w \rangle$ denotes a Euclidean product between vectors $z$ and $w$.

The remainder of the paper is structured as follows. Section \ref{sec:reg} 
describes a framework for pseudo-time regularization, exploiting the 
quasi-linear structure of \eqref{eqn:weak}, and specifies the regularized system
under inexact assembly.  Section \ref{sec:residual_rep} derives
the expansion of the latest residual in terms of the previous residual, exposing
the regularization and linearization error terms.  Section \ref{sec:param_updates} 
suggests a set of regularization parameter updates 
based on the representation of Section~\ref{sec:residual_rep}, together with
an adaptive algorithm.
Finally, Section \ref{sec:numerics} demonstrates the ideas with some numerical
experiments. 

\section{Regularized formulation}\label{sec:reg}
The regularization framework is next described.  First, an error decomposition
is discussed to separate out the contributions to the error in each iterate; 
that is, each solution
to a linearized problem used to approximate the solution to \eqref{eqn:weak}.
The contributions to the error include those induced from added regularization, 
linearization, early termination of the iterations, inexact assembly, 
and discretization.

Then, in Section~\ref{subsec:reg_abstract}, 
a pseudo-time regularized iteration is described and fit into this framework. 
First, the pseudo-time regularization is introduced in a general sense, then the 
regularized iteration is derived from a linearization of the abstract formulation 
of the discrete problem~\eqref{eqn:abs_discrete}.  Section~\ref{subsec:assembly}, 
then introduces notation for the inexact assembly of the discrete problem, leading to
the regularized iteration in matrix form.  The steps are separated out in this
presentation to emphasize the relation between the practical inexactly assembled
equation that is actually solved computationally, and the abstract formulation
of the problem that is thought of on the PDE level.  
\subsection{Error decomposition}\label{subsec:error_decomp}
The goal of the numerical method, finite element or otherwise, is to approximate
the PDE solution; in this case, the solution to~\eqref{eqn:weak}. 
The error may be understood by breaking it into components
describing the discretization error, quadrature error, and so-called linearization or
nonlinear iteration error.  In the following notation, $u$ represents the (or a) 
solution to~\eqref{eqn:weak}; 
and $u_k^{\cE,\ast}$ denotes the solution to the discrete nonlinear problem on 
the $k^{th}$ refinement of the initial mesh partition, under exact assembly.  
Neither $u$ nor $u_k^{\cE, \ast}$ are generally computable quantities. 
The iterate $u_k^\ast$ is the solution up to some set tolerance of the discrete 
problem using a numerical and potentially inexact assembly procedure; $u_k^\ast$ may 
or may not be computable. 
In terms of the inexactly assembled problem, 
$u_k$ is the terminal iterate, potentially before 
convergence to tolerance, and $u_k^n$ is the $n^{th}$ 
iterate on the $k^{th}$ refinement.
\begin{align}\label{eqn:error_exp001}
u - u_k^n = (u - u_k^{\cE,\ast}) + (u_k^{\cE,\ast} - u_k^\ast) + (u_k^\ast - u_k) + 
            (u_k - u_k^n).
\end{align}
The frameworks developed in~\cite{BiGi13,HTZ09a} show convergence to zero of
the first term of \eqref{eqn:error_exp001}, for a restricted class of problems of
the type~\eqref{eqn:PDE} (see Remark~\ref{remark:known_classes}),
under the assumptions of a sufficiently small meshsize.  
These asymptotic results motivate the current work, by demonstrating
the approximation properties of a finite element solution to the PDE solution; 
however, the goal now is to develop a computational framework for which 
$u_k^{\cE,\ast} \goto u$ still holds, and by which $u_k^{\cE,\ast}$ can be approximated
by a computable sequence.

To that end, a regularized iteration is introduced starting on the initial,
presumably coarse, mesh.  The expansion of the error incorporating the 
introduced regularization breaks the last term of~\eqref{eqn:error_exp001} into
two new terms.
Let $u_k^S$ be the terminal iterate of the regularized problem on refinement $k$.  
For ease of notation, the indices on $S$ are suppressed, however it should be 
understood that $u_k^S$ is subject to regularization $S_k$; and likewise, on
iteration $n$ of refinement $k$, the iterate $u_k^{S,n}$ is subject
to regularization $S_k^n$.
\begin{align}\label{eqn:error_exp002}
u - u_k^{S,n} = (u - u_k^{\cE,\ast}) + (u_k^{\cE,\ast} - u_k^\ast) + (u_k^\ast - u_k) + 
            (u_k - u_k^S) + (u_k^S - u_k^{S,n}).
\end{align}
The regularization addressed in this paper can be broken into two parts: a
Jacobian regularization and a residual regularization term.  
Then the regularized Newton-like iteration described in the following sections
can be written in terms of a standard
Newton iteration for this problem with Jacobian $J = J(u)$ and residual $r = r(u)$, 
$J w = r, ~ u \leftarrow u + w$, by
\begin{align}\label{eqn:reg_structure}
(J + S^J)w = r + S^{res}, \quad u \leftarrow u + w. 
\end{align}
Denoting $S = S^J + S^{res}$, eliminating the regularization from the iteration
can be described as sending $S \goto 0$.

Control of the last term of \eqref{eqn:error_exp002}, 
$(u_k^S - u_k^{S,n})$ requires assumptions on 
the PDE, for instance Assumption~\ref{assume:pdeN}, as well as assumptions on the 
regularization; essentially, the sequence of regularized linear problems must be 
sufficiently stable.

Due to the regularization structure~\eqref{eqn:reg_structure}, 
control of the fourth term $(u_k - u_k^S)$ is established by $S \goto 0$ for 
all $k \ge K_0$, for some iteration $K_0$.  If the problem~\eqref{eqn:weak} 
is sufficiently well-posed, then the regularized iteration should 
limit and indeed revert to a standard Newton-iteration on each refinement after
some level $K_0$.  This type of result is detailed for a similar regularized 
Newton-like method in \cite{Pollock15b}. It is noted however that sending 
$S^{res} \goto 0$ restores consistency of the iteration. If a problem
features a potentially indefinite or badly conditioned Jacobian in the 
vicinity of a solution, it may be beneficial not to send $S^J$ to zero, rather
to keep some background level of Jacobian regularization that does not interfere 
with the convergence of the iteration.

The third term of \eqref{eqn:error_exp002} is the difference between 
the terminal iterate and one converged to tolerance.  This term is zero 
for $k$ larger than some $K_1$, where $K_1$ may be close and potentially equal to $K_0$,
under a suitable residual-reduction condition for terminating the nonlinear 
iterations on each refinement level $k$.
In the early stages of the solution process however, 
the nonlinear iteration may be stopped
far from convergence, and the difference $(u_k^\ast - u_k)$ may be non-negligible.

The second term of~\eqref{eqn:error_exp002} describes the error induced by 
inexact integration. 
With the potential of a highly oscillatory diffusion coefficient in \eqref{eqn:PDE},
this error term is not automatically assumed to be small as it is 
in~\cite{BiGi13,HTZ09a} and the references therein, 
where the assumption of a sufficiently small global meshsize can
control the level of accuracy.  Generally, information can 
be lost both by the averaging of integration against basis functions, and by 
the secondary averaging of approximate integration by quadrature.  
Finally, the first term $u - u_k^{\cE,\ast}$ constitutes the discretization error,
the difference between a solution to~\eqref{eqn:weak} and the solution of the 
exactly assembled discrete nonlinear problem on refinement $k$.  
The analysis of these first two terms, determining convergence of the error, 
is beyond the scope of this paper, which analyzes efficient convergence of the
residual.  The conditions under which $u_k^{\ast}$ converges to $u$ for $k$
greater than some $K_2$ will be discussed elsewhere.  

\begin{remark}[Inexact linear solves]\label{remark:inexact_linear_solves}
Solving the linear equations by an iterative method yields yet another term in 
the expansion.  For linear iteration $l$, the error between the PDE solution $u$
and iteration $l$, of regularized linear solve $n$, on mesh refinement $k$, 
may be decomposed into 
\begin{align}\label{eqn:error_exp003}
u - (u_k^n)^l &= (u - u_k^{\cE,\ast}) + (u_k^{\cE,\ast} - u_k^\ast) + (u_k^\ast - u_k) + 
            (u_k - u_k^S) + (u_k^S - u_k^{S,n})
\nonumber \\
         & + ( u_k^{S,n} - (u_k^{S,n})^l). 
\end{align}
The last error term is not addressed in the current paper, and the linear systems
are assumed solved exactly.  It is noted however that incomplete linear solves can
be exploited for both their regularization properties and efficiency, 
and this topic is worth investigation.
\end{remark}

\subsection{Regularized abstract formulation}\label{subsec:reg_abstract}
The regularized linear equations compatible with error decomposition
\eqref{eqn:error_exp002} are now derived through a pseudo-time discretization 
framework with respect to the abstract discrete problem \eqref{eqn:abs_discrete}.
The pseudo-time framework developed in for instance \cite{BaRo80a,CoKeKe02,KeKe98},
the book~\cite{Deuflhard11}, and the references therein, suggests to stabilize 
the solution of the elliptic equation $F(u) = 0$, introduce the pseudo-time 
dependence to $u$, and solve $\pa u/\pa t  + F(u) = 0$. To allow a preconditioned
or more general regularized framework, the discrete nonlinear pseudo-time regularized
problem is: Find $u \in \cV_k \subset \cV$ such that
\begin{align}\label{eqn:reg_001}
\phi(\dot u,v) + B(u; u,v) = \int_{\Omega} fv, \tforall v \in \cV_k,
\end{align}\label{eqn:reg_cts}
where $\dot u = \pa u/\pa t$.
Here, it is assumed that the bilinear form
$\phi(\cdot\, , \cdot \, )$ is continuous with respect to 
the native norm $\cV$, in this case taken to be the $H^1$ norm.
\begin{align}\label{reg_semidef}
\phi(w,v) \le C_\phi\nr{w}_1 \nr{v}_1.
\end{align}
The continuity assures $\phi(\dot u, \cdot)$ decays to zero as $\dot u \goto 0$
indicating a steady state solution. 
Rather that coercive, $\phi$ is assumed semi-definite
\begin{align}
\phi(w,w) \ge 0.
\end{align}
This allows for a more general regularization, as used in~\cite{Pollock14a,Pollock15a},
in which a cutoff function is used to only allow the regularization to act on 
select degrees of freedom.  Upon matrix assembly, the role of $\phi$ can be viewed
as improving the condition of the approximate Jacobian, and it should be chosen 
with this in mind. 
In the numerical experiments of Section \ref{sec:numerics},
two different choices of $\phi$ are illustrated.  The first sets
$\phi(w,v) = ((1+|\kappa'(z)|) \grad w, \grad v)$, for $z$ the initial solution iterate
on each refinement, thus adding more regularization locally to control steeper gradients
in the diffusion.  The second uses
$\phi(w,v) = (\grad w, \grad v)$, the standard Laplacian preconditioner, adding a 
uniform level of diffusion to stabilize the Jacobian.
The regularization functional is left in general form for the remainder of the analysis 
to emphasize that these are two of many choices.

A generalization of the Newmark time integration strategy ~\cite{Newmark59}, exploiting
the structure of the quasilinear equation~\eqref{eqn:PDE} is now introduced to
discretize \eqref{eqn:reg_001} in pseudo-time
\begin{align}\label{eqn:reg_002}
\phi({(\Delta t^n})^{-1} w^n, v) = 
   -\widetilde \gamma_{00} B(u^n; u^n, v) 
   -\widetilde \gamma_{10} B(u^{n+1}, u^n, v)
   -\widetilde \gamma_{01} B(u^n, u^{n+1}, v) 
   + (f,v),
\end{align} 
where $(f,v) = \int_{\Omega}f v,$ and $w^n = u^{n+1} - u^n$. 
Linearizing the second term on the right, and rewriting the third to isolate the
dependence on $w^n$ yields 
\begin{align}\label{eqn:reg_003}
\phi({(\Delta t^n})^{-1} w^n, v) &= 
   -\widetilde \gamma_{00} B(u^n; u^n, v) 
   -\widetilde \gamma_{10} \left( B(u^{n}, u^n, v) + B_1'(u^n; u^n,v)(w^n) \right) 
   \nonumber \\
   &\quad -\widetilde \gamma_{01} \left( B(u^n, u^{n}, v) + B(u^n, w^n, v) \right)
   + (f,v),
\end{align} 
where $B_1'(u; z, v)(w) \coloneqq \f {d}{ds} B(u + sw; z, v)\rest_{s = 0}$, 
the Gateaux derivative in the first argument of $B$, in the direction $w$.
Rearranging \eqref{eqn:reg_003} so that all terms involving the update step
$w^n$ appear on the left, and rescaling by 
$\widetilde \gamma \coloneqq \widetilde \gamma_{00} + \widetilde \gamma_{10} + 
\widetilde \gamma_{01},$ yields
\begin{align}\label{eqn:reg_004}
\alpha^n \phi(w^n, v) + \gamma_{10}B_1'(u^n;u^n,v)(w) + \gamma_{01} B(u^n; w^n,v)
= -B(u^n; u^n, v) + \delta (f,v), 
\end{align} 
with the four regularization coefficients given by
\begin{align}\label{eqn:reg_005}
\alpha^n     = 1 /( {\Delta t^n} \cdot {\widetilde \gamma}), \quad
\gamma_{10}  = {\widetilde \gamma_{10}}/ {\widetilde \gamma}, \quad
\gamma_{01}  = {\widetilde \gamma_{01}}/ {\widetilde \gamma}, \quad \an \quad
\delta       = 1/ {\widetilde \gamma}.
\end{align}
The coefficient $\alpha^n$ is then the rescaled reciprocal of the pseudo-time step, 
and $\alpha^n \goto 0$, corresponds to $\Delta t^n \goto \infty$.  It is
remarked that $\gamma_{10} = \gamma_{01}$ corresponds to the method discussed in
\cite{Pollock15a}, where this parameter is taken greater than one to introduce 
an increase in numerical dissipation, or controlled damping, into the iteration.
Further, $\gamma_{01} = 1 = \gamma_{10}$ corresponds to an implicit, or backward Euler
discretization of \eqref{eqn:reg_001}; while $\gamma_{01} = 0 =\gamma_{10}$, 
corresponds to an explicit, or forward Euler discretization of \eqref{eqn:reg_001}.  
And finally, $\gamma_{01} = 1$ and $\gamma_{10} = 0$ leads to a Picard iteration.
It is also recognized that $\delta = 1$ yields a consistent pseudo-time 
discretization.  

\subsection{Inexact assembly}\label{subsec:assembly}
Both the finite dimensional equation~\eqref{eqn:abs_discrete}, and the pseudo-time
regularized \eqref{eqn:reg_001}, are abstract
equations rather than computable systems.  
To clarify which quantities are assumed computationally available, 
the following notation is introduced to 
describe the discrete system induced inexact assembly, {\em e.g}
by quadrature, which may be assumed inexact for nonpolynomial integrands. 
Let $\cV$ be a discrete space, here a finite element space, with $n_{dof}$
degrees of freedom, spanned by the basis functions $\{\varphi^j\}_{j = 1}^{n_{dof}}$.
Supposing $u,v,w,z \in \cV$, each function has an exact expansion as a linear
combination of basis functions; in particular
$
w = \sum_{j = 1}^{n_{dof}} \tt w_j \varphi^j,
$
with $\tt w $ the vector of coefficients $\w_j, ~j = 1, \ldots, n_{dof}$.
Let $A(u;z)$ be an inexact assembly of 
$B(u;z,\varphi^j), ~j = 1, \ldots, n_{dof}$, with an error introduced by inexact 
integration, {\em e.g.,} quadrature error.
The source vector $f_{Q}$ is formed by the inexact integral of source function $f$
against each basis function $\varphi^j$.
The matrix assembly of the regularization $\phi(w,\varphi^j)$ is denoted $R$.
Let $\cA_Q$ represent the inexact
assembly operator.  The assembled systems under $\cA_Q$ are denoted
as follows.
\begin{align}
A(u;z)     & \coloneqq \cA_Q \left\{ (B(u;z,\varphi^j) 
             \right\}_{j = 1}^{n_{dof}}, \label{eqn:defA}\\
A_1'(u;z)\tt w & \coloneqq \cA_Q \left\{\f{d}{dt} B(u+tw;z,\varphi^j)
            \rest_{t=0} \right\}_{j = 1}^{n_{dof}}, 
\label{eqn:defA1p}\\
A_2'(u)\tt w   &\coloneqq \cA_Q \left\{\f{d}{dt} B(u;z+tw,\varphi^j)
           \rest_{t=0} \right\}_{j = 1}^{n_{dof}} = A(u,w), \label{eqref:defA2p}\\
R\w     & \coloneqq \cA_\cE  \left\{ \phi(w,\varphi^j) 
               \right\}_{j = 1}^{n_{dof}}, \label{eqn:defRi}\\
f_Q  & \coloneqq \cA_Q \left\{ \int f \varphi^j \right\}_{j = 1}^{n_{dof}}. 
        \label{eqn:defF}
\end{align}
The following commuting diagram  holds for the discrete assembly 
procedure given by \eqref{eqn:defA}-\eqref{eqn:defA1p}, with $u,v,w,z \in \cV$.
That is, the inexact assembly operator $\cA_Q$ commutes with the  Gateaux
derivative of the first argument of $B(\cdot \, ,  \cdot \, ,  \cdot \, )$.  
\begin{align}\label{eqn:commute_quad}
B(u;z,v) && \underrightarrow{\qquad\cA_Q\qquad}      && A(u;z) 
\nonumber \\
{\small \pa_u} \Bigg\downarrow &&            &&{\small \pa_u}\Bigg\downarrow 
\nonumber \\
B_1'(u;z,v)(w) && \underrightarrow{\qquad\cA_Q\qquad} && A_1'(u;z)\tt w,
\end{align}
This justifies the use of Taylor's theorem
in the error representation of the residual in Section~\ref{sec:residual_rep}.
In general, the Gateaux derivative commutes with 
projection-type discretizations; see for example~\cite{HoBaWa00}.  
This includes assembly under inexact integration, 
assuming the integral approximation  over each element $T$
falls in the general form  
$\int_T \phi \approx \sum_{i = 1}^{n_{Q_T}} \phi(x_i) \rho_i$, 
for $n_{Q_T}$ points $x_i$ in the interior of element $T$,  and weights $\rho_i$.

\subsection{Regularized matrix equations}\label{sec:reg_matrix}
Processing the linear pseudo-time regularized equation \eqref{eqn:reg_004} with the 
inexact assembly given by~\eqref{eqn:defA}-\eqref{eqn:defF}, yields the coefficients
of the update step $w$ as the solution to a linear system of equations.
\begin{align}\label{eqn:iteration_mainQ}
\left\{ \alpha^n R + \gamma_{10}A_1'(u^n;u^n) + \gamma_{01}A_2'(u^n)\right\} \w^n
  = \delta f_Q - A(u^n;u^n),
\end{align}
which may be written as
\begin{align}\label{eqn:iteration_Q}
M^n \w^n &= \f 1 {\gamma_{10}} r^n, ~\text{ with}~ \\
M^n     & = \f 1{\gamma_{10}}\alpha^n R + A_1'(u^n;u^n)  
               + (1 + \sigma_{01}^n)A_2'(u^n), \label{eqn:iteration_Mn} 
\quad \sigma_{01} = \f{\gamma_{01}}{\gamma_{10}} - 1, 
\\ 
r^n     & = \delta f_Q - A(u^n;u^n), \label{eqn:iteration_rn}
\end{align}
with the update $u^{n+1} = u^n + w^n$.

With respect to~\eqref{eqn:reg_structure}, 
the formal representation of the regularization structure, the Jacobian part $S^J$,
and the residual part $S^{res}$, of the regularization are given by
\begin{align}\label{eqn:reg_formal}
S^J & = \alpha R + (\gamma_{10}-1) A_1'(u^n; u^n) + (\gamma_{01}-1)A_2'(u^n), ~\an \\
S^{res} & = (\delta -1) f_Q.  
\end{align}
Consistency is restored by sending regularization parameter
$\delta \goto 1$.  Asymptotic efficiency is restored by sending $\alpha \goto 0$, 
$\gamma_{10} \goto 1$ and $\gamma_{01} \goto 1$, although this asymptotic 
efficiency may be at least partially sacrificed for stability, even into the
asymptotic regime, where the iterations converge to tolerance.  This balance
is understood in the next section where the residual representation exposes
the error contributions from regularization and linearization. So long as 
the regularization effectively controls the linearization error without increasing 
the norm of the residual, it is viewed as beneficial.
\section{Residual representation of the matrix equation}
\label{sec:residual_rep} 
The residual representation follows the standard method of applying Taylor's 
theorem to expand the $(n+1)^{th}$ residual about the $n^{th}$ residual, 
justified by the commuting diagram~\eqref{eqn:commute_quad}.  
This exposes the separate terms from the introduced regularization 
error and the intrinsic linearization error.  
The linearization error is bounded by a Lipschitz assumption 
on the problem data \eqref{eqn:LipKprime}, 
although approaches with more general assumptions such as a majorant condition
have also been developed for Newton iterations~\cite{Ferreira11}, and
would be interesting to investigate in the present context.
Unlike previous presentations by the author
\cite{Pollock14a,Pollock15a,Pollock15b}, here the structure of the
quasilinear problem is exploited to separate the linear and nonlinear
dependencies on the latest iterate $u^n$.
A choice of regularization parameters in then introduced in the context 
of minimizing an appropriate quantity to control the linearization error.

\subsection{Residual representation under inexact integration}
\label{subsec:resiRep_inexact}
Expanding the residual $r^{n+1}$ about $r^n$ yields
\begin{align}\label{eqn:error001}
r^{n+1}& = \delta f_Q - A(u^{n+1};u^{n+1}) 
\nonumber \\
       & = \delta f_Q - A(u^{n+1}; u^n) - A(u^{n+1}; w^n) 
\nonumber \\
       & = r^n - A_1'(u^n;u^n)w^n - A(u^n;w^n) 
           -K_1 - K_2,
\end{align}
with
\begin{align}\label{eqn:comperr_002}
K_1 & \coloneqq 
      \int_0^1\left\{ A_1'(u^n + t w^n;u^n) - A_1'(u^n;u^n) \right\}\w^n\, dt
\nonumber \\  
    & = A(u^{n+1};u^n) - A(u^n; u^n) - A_1'(u^n; u^n)\w^n, \\
K_2 & \coloneqq  \int_0^1 A_1'(u^n + tw^n; w^n)\w^n\, dt 
      = A(u^{n+1}; w^n) - A(u^n; w^n). 
\end{align}
Solving~\eqref{eqn:iteration_Q} for $A_1'(u^n; u^n)\w^n$ yields
\begin{align}\label{eqn:comperr_002a}
-A_1'(u^n;u^n)\w^n &=  \f{1}{\gamma_{10}}\alpha^n R \w^n  
                     +\f{\gamma_{01}}{\gamma_{10}}A(u^n;w^n) 
                     - \f{1}{\gamma_{10}}r^n + \fl,
\end{align}
where the floating-point arithmetic 
error $\fl$ is introduced from the solution of the linear system for 
coefficients $\w^n$. Applying \eqref{eqn:comperr_002a} to 
\eqref{eqn:error001} yields
\begin{align}\label{eqn:comperr_003}
r^{n+1} & = \left( 1 - \f 1{\gamma_{10}} \right)r^n 
         + \f{1}{\gamma_{10}}\alpha^n R\w^n
         + \sigma_{01}A(u^n;w^n)
         +\lin(u^n) + \fl. 
\end{align}
Here $K_1$ describes the dominant term in the linearization error, and 
$K_2$ the secondary term, whose linear component is $w^n$.  
The total linearization error is defined as
\begin{align}\label{eqn:comperr_004}
\lin(u^n) \coloneqq 
-(A(u^{n+1}; u^{n+1}) - A(u^n; u^{n+1}) - A_1'(u^{n};u^{n+1})\w^n) 
= -K_1 - K_2, 
\end{align}
which agrees with the definition of the one-step linearization error
given in~\cite{Pollock15b}.   
The convergence of the residual then follows from control over the 
linearization error, assuming the floating point error is 
sufficiently negligible.  A remark about this terms follows.
\begin{remark}[Floating-point error]\label{remark:floatingpoint}
The last
term in~\eqref{eqn:comperr_003}, $\fl$, denotes the floating point error, 
which cannot be controlled by the linearization, but neither can it be
entirely ignored.  
It can be estimated, for instance by the difference between two evaluations 
of the linearization error $\lin(u^n)$, one by \eqref{eqn:comperr_004}, and
the other by isolating $\lin(u^n) + \fl$ in \eqref{eqn:comperr_003}. 
In the preasymptotic and coarse mesh regimes, where 
the iterate $u^n$ is sufficiently far from the solution, the floating point
error, observed in the numerical experiments in Section~\ref{sec:numerics} 
remains on the order of $10^{-10}$ to $10^{-12}$.  However approaching the
asymptotic regime as $\nr{w^n}$ approaches $10^{-5}$ or $10^{-6}$, the linearization 
error is no longer observed to be $\bigo(\nr{w}^2)$, even where analytically
it should be. In this regime the linearization error is 
$\nr{\lin(u^n)} = \bigo(\nr{w^n})$, due to the pollution from the floating-point error. 
In terms of practical impact on a computational method such as the one described here, 
$\fl$ limits the regime
where the convergence rate can be accurately detected.  This is immaterial,
so long as detecting that convergence rate is no longer necessary once, for instance,
$\eps_T\nr{r^n} = \bigo(\fl)$, where $\eps_T$ is a set tolerance.
\end{remark}

The control of the right-hand side linearization error $\lin(u^n)$
is left to the choice of regularization terms $\alpha, \gamma_{10},  
\sigma_{01}$ and $\delta$.  It is remarked that $\lin(u^n)$ does 
not necessarily need to be second order with respect to $w^n$ for convergence
of the method: it only needs to be small enough not to interfere with 
the convergence rate.

Local convergence theory for Newton-like methods describes the convergence
of the iterates in a neighborhood near the solution, and is addressed
for regularized pseudo-time algorithms by the author in previous work
~\cite{Pollock14a,Pollock15a,Pollock15b}, based in part on the analysis of
\cite{BaRo80a,CoKeKe02,KeKe98} and \cite{Deuflhard11}.  In practice, however, 
the predicted convergence of rate of iteration 
\eqref{eqn:iteration_Q}-\eqref{eqn:iteration_rn}
is often oberved from the first few iterations
without a particularly good initial guess. Here the goal is to
characterize the convergence rate when the iterate $u^n$ is not sufficiently
close to the solution $u^\ast$ of $A(u;u) = f_Q$.
\begin{lemma}[Convergence rate far from the solution] 
\label{lemma:farconvergence}
Consider iteration~\eqref{eqn:iteration_mainQ} applied to the 
problem $A(u;u) = f_Q$.  Assume the regularization parameters,
$\alpha^n$ and $\sigma_{01}^n$,  satisfy the
following properties.
\begin{align}
\alpha^n \cdot \f 1 {\gamma_{10 }} \nr{R\w^n} 
& \le \f{\eps_T}{2} \nr{r^n}, \label{assume:alpha} \\
\nr{ \sigma_{01}^n A(u^n;w^n) + \lin(u^n) + \fl} 
& \le \nr{\lin(u^n) + \fl}. \label{assume:sigma}
\end{align}
Then 
\begin{align} \label{eqn:far001}
\nr{r^{n+1}} \le \left(1 - \f 1 {\gamma_{10}}\right) \nr{r^n} + \f {\eps_T}{2}
+  \nr{\lin(u^n) +  \fl}.
\end{align}
Moreover, if it holds that 
\begin{align}\label{eqn:far002}
\f{\nr{\lin(u^n) + \fl}}{\nr {r^n} } < \f{\eps_T}{2},
\end{align}
then the iteration~\eqref{eqn:iteration_mainQ} converges within tolerance
$\eps_T$ of the predicted rate $(1 - 1/\gamma_{10})$.
\end{lemma}
\begin{proof}
The bound~\eqref{eqn:far001} follows directly from applying hypotheses
\eqref{assume:alpha} and \eqref{assume:sigma} to the residual 
representation~\eqref{eqn:comperr_003}.  
\end{proof}
Ultimately, residual convergence at the predicted rate comes down to whether the
linearization error $\lin(u^n)$ can be controlled.  The following discussion 
investigates when this is computationally reasonable.  
Based on Assumption~\ref{assume:pdeN} on the problem data $\kappa'$, there exist
positive constants $\omega_A$ and $C_A$ with
\begin{align}\label{eqn:overLip}
\nr{(A_1'(u^n + tw^n; u^n) - A_1'(u^n; u^n))w^n} & \le t\, \omega_A \nr{\w^n}\nr{w^n}, 
    ~t > 0\\ \label{eqn:overBound}
\nr{A_1'(u^n; w^n; w^n)} & \le C_A \nr{\w^n}\nr{w^n}.
\end{align}
Applying \eqref{eqn:overLip}-\eqref{eqn:overBound} to the linearization error 
given by \eqref{eqn:comperr_004}, one obtains
\begin{align}\label{eqn:far003}
\nr{\lin(u^n)} & \le \left\| \int_0^1 (A_1'(u^n + t w^n;u^n) - A_1'(u^n;u^n))\w^n\, dt 
\right\|  + \left\| \int_0^1 A_1'(u^n + tw^n; w^n)\w^n \, dt \right\|\nonumber \\
         & \le \left( \f{\omega_A}{2} \nr{w^n} + C_A \nr{w^n } \right) \nr{\w^n}. 
\end{align}
Then from iteration \eqref{eqn:iteration_Q}
\begin{align}\label{eqn:far004}
\f {\nr{\lin(u^n)}}{\nr{r^n}}  
 \le \f{1}{\gamma_{10}}
\left( \f{\omega_A}{2} + C_A \right) \nr{w^n} \f{\nr{\w^n}}{\nr{M^n \w^n}}. 
\end{align}
While estimate \eqref{eqn:far004} is true, and it illustrates the role of $\gamma_{10}$
as a damping parameter, it may greatly overestimate the linearization error and
is not useful as a predictor of when \eqref{eqn:far002} will hold.

The source of the overestimate in this context is allowing for the maximum 
Lipschitz constant and bound on $\kappa'$ to be achieved uniformly 
over the domain.  Standard adaptive finite element methods are known to perform
well with relatively few local high contrast heterogeneities or singularities,
but are not necessarily appropriate for globally high contrast domains or coefficients, 
so it makes sense to understand how local high contrast can effect the convergence.
Writing $\lin(u^n)$ in terms of the inexact assembly operator $\cA_Q$
\begin{align}\label{eqn:far005}
\lin(u^n) &= -\int_{t = 0}^1 \cA_Q \left\{ 
   \int_{\Omega} (\kappa'(u^n + tw^n) - \kappa'(u^n))w^n \grad u^n \cdot 
   \grad \varphi^j  \right\}_{j = 1}^{n_{dof}} dt  \nonumber \\
  &- \int_{t = 0}^1 \cA_Q \left\{ 
   \int_{\Omega} \kappa'(u^n + tw^n)w^n \grad w^n \cdot
   \grad \varphi^j  \right\}_{j = 1}^{n_{dof}} dt.
\end{align}
Freezing the analysis about the iterate $u^n$, expression \eqref{eqn:far005} 
suggests partitioning $\Omega$ into $\Omega_n$, 
where milder bounds than $\omega_A$ and $C_A$ are realized, and 
$\Omega_n^C = \Omega \setminus \{\Omega_n\}$, where these bounds are locally 
attained.  
Rewriting \eqref{eqn:far005} in terms of a partition $\Omega = \Omega_n \cup \Omega_n^C$
\begin{align}\label{eqn:far006}
\nr{\lin(u^n) } \le \nr{\lin_1(u^n)} + \nr{\lin_2(u^n)},
\end{align}
with
\begin{align}\label{eqn:far007}
\lin_1(u^n) &= -\int_{t = 0}^1 \cA_Q \left\{ 
   \int_{\Omega_n} (\kappa'(u^n + tw^n) - \kappa'(u^n))w^n \grad u^n \cdot 
   \grad \varphi^j  \right\}_{j = 1}^{n_{dof}} dt  \nonumber \\
  &- \int_{t = 0}^1 \cA_Q \left\{ 
   \int_{\Omega_n} \kappa'(u^n + tw^n)w^n \grad w^n \cdot
   \grad \varphi^j  \right\}_{j = 1}^{n_{dof}} dt, \an \\ \label{eqn:far008}
\lin_2(u^n) &= -\int_{t = 0}^1 \cA_Q \left\{ 
   \int_{\Omega_n^C} (\kappa'(u^n + tw^n) - \kappa'(u^n))w^n \grad u^n \cdot 
   \grad \varphi^j  \right\}_{j = 1}^{n_{dof}} dt  \nonumber \\
  &- \int_{t = 0}^1 \cA_Q \left\{ 
   \int_{\Omega_n^C} \kappa'(u^n + tw^n)w^n \grad w^n \cdot
   \grad \varphi^j  \right\}_{j = 1}^{n_{dof}} dt.
\end{align}
From the data Assumption \ref{assume:pdeN} and the decomposition 
\eqref{eqn:far006}-\eqref{eqn:far008}, for each partition of the domain into
$\Omega_n$ and $\Omega_n^C$ there is a smallest constant 
$K_n(\Omega_n) \le (\omega_A/2+ C_A)$ with 
\begin{align}\label{eqn:far009}
\nr{\lin_1(u^n)} &\le \meas(\Omega_n) \cdot {K_n}  
                      \nr{w^n} {\nr{\w^n}}, \an \\ \label{far:010}
\nr{\lin_2(u^n)} &\le \meas(\Omega_n^C)\cdot {K_A} 
                      \nr{w^n} {\nr{\w^n}},
\end{align}
where $K_A \le \omega_A/2 + C_A$. With this structure in place, it follows that
the condition \eqref{eqn:far002} holds if there is a partition $\Omega_n$ for which
\begin{align}\label{eqn:far011}
\f 1 {\gamma_{10}}\left( \meas(\Omega_n) \cdot K_n 
                       + \meas(\Omega_n^C) \cdot K_A \right)\nr{w^n} 
\f{ \nr{\w^n} }{ \nr{M^n \w^n }} \le \f {\eps}{2},
\end{align}
for a given $\eps$.

The ratio $\nr{\w^n}/\nr{M^n \w^n} = \nr{(M^n)^{-1}r^n} / \nr{r^n}$ 
is related to the condition of the 
approximate Jacobian $M^n$ given by \eqref{eqn:iteration_Mn}, and is explicitly
dependent on the parameters $\gamma_{10}, \sigma_{01}$ and $\alpha$, as well
as implicitly dependent upon $\delta$. A large parameter $\gamma_{10}$ can clearly
control the scale of the linearization error at the start of the adaptive algorithm,
and if the steep gradients are bounded away from zero, a small scaling parameter
$\delta$ can control $(\meas(\Omega_n) \cdot K_n + \meas(\Omega_n^C) \cdot  K_A)$.
However, to attain convergence of the residual with $\gamma_{10} = 1$, in some
computationally available neighborhood of the solution $u^\ast$ to $A(u,u) = f_Q$,
the measure of the set on which a large Lipschitz constant and bound on the first 
derivative of $\kappa$ is realized must be relatively small. Otherwise, 
${\nr w}$ may need to  be  small enough that it is computationally infeasible
to find the basin of attraction.

A choice of regularization parameters is next described with respect to
the numerically assembled iteration~\eqref{eqn:iteration_Q}.
In particular, the parameter $\sigma_{01}$ guiding the Picard-like regularization
is based on the condition \eqref{assume:sigma}; and, the Tikhonov-like
regularization scaled by $\alpha$ is based on the condition \eqref{assume:alpha},
from Lemma \eqref{lemma:farconvergence}.

\section{Regularization parameter updates}\label{sec:param_updates}
A set of regularization parameters $\alpha, ~\gamma_{10}, \gamma_{01}$, and $\delta$
is now presented, along with a discussion of their properties.
The definition of $\gamma_{10}$ is consistent with $\gamma$ given in 
\cite{Pollock15b}, as is the definition of $\delta$.  A different definition of the
parameter $\alpha$ is given here, than in \cite{Pollock14a,Pollock15a,Pollock15b},
and the parameter $\gamma_{01}$ has not been previously introduced.  It is noted, 
however, that $\gamma_{01}$ effectively adds diffusion to the linearized system
by adding a Picard-like term to the Newton-like iteration.  
In \cite{Pollock15a}, the parameter $\sigma$ adds a frozen Newton-like iteration 
to stabilize the approximate Jacobian, essentially preventing small eigennvalues from
changing sign at each step.  The new parameter $\gamma_{01}$ performs a similar role,
but is more amenable to analysis, and appears to perform better in numerical 
experiments.

\subsection{Update of numerical dissipation, $\gamma_{10}$} 
\label{subsec:gamma10}
The first order regularization error is controlled by the numerical 
dissipation parameter, {\em i.e., } the Newmark parameter, $\gamma_{10}$.
The definition used here is recalled from \cite{Pollock15b}, Definition 4.3;
and framed in the context of an $L_2$ minimization as follows.
Rewriting \eqref{eqn:comperr_003} by moving the residual terms involving $r^n$ to the 
left-hand side and taking the $L_2$ norm of both sides of the resulting equation
\begin{align*}
\left\| (r^{n+1} - r^n) + \f 1{\gamma_{10}}r^n \right\| 
         = \left\| \f{1}{\gamma_{10}}\alpha^n R\w^n
         + \sigma_{01}A(u^n;w^n) 
         +\lin(u^n) + \fl \right\|.       
\end{align*}
An updated value of $\gamma_{10}$ is chosen to minimize the norm on the left,
namely
\begin{align*}
\f 1{\widetilde \gamma_{10}} = \argmin_{\nu \in \R}
\nr{(r^{n+1} - r^n) + \nu r^n}
&= \f {\langle r^n, r^{n} - r^{n+1}\rangle }{\nr{r^n}^2} \\ 
&= \f{\langle r^n, A(u^{n+1}; u^{n+1}) - A(u^n; u^n)\rangle}{\nr{r^n}^2}.
\end{align*}
The update of $\gamma_{10}$ is then defined by
\begin{align}\label{eqn:update_gamma10}
\tilde\gamma_{10} \coloneqq  
\f{\nr{r^n}^2}{\langle r^n, r^n - r^{n+1}\rangle}, \an  
\gamma_{10} \leftarrow \max\left\{ q \cdot \widetilde \gamma_{10}, 1 \right\},
\end{align}
for a user-set parameter $q$, with $0 < q < 1$.

The purpose of introducing the parameter $q$ is to enforce monotonicity of 
the sequence of parameters $\{\gamma_{10}^n\}$ to one at a given 
rate.  As shown in \cite{Pollock15b}, if $\gamma_{10}$ is updated when 
the residual reduction satisfies the following condition, then there is 
a critical value $\gmo$, after which $\gamma_{10}$ is assured to reduce
at a linear rate.  Those results are included in the following more general
lemma, which features a condition on the direction cosine of consecutive 
residuals to determine predictable reduction of $\gamma_{10}$.

Let $\gamma_{10,k}^n$ be the value of $\gamma_{10}$, 
on iteration $n$ of refinement $k$.  
For simplicity of notation, $\gamma_{10,k}^n$ will be denoted as $\gamma_{10}^n$.
For the update \eqref{eqn:update_gamma10} 
to remain bounded under the conditions that follow, 
the parameters $\GM$, the maximum allowed value of $\gamma_{10}$, and $\eps_T$, 
the rate tolerance used to determine whether $\gamma_{10}$ should be updated, 
are now introduced to satisfy the following condition.
\begin{condition}\label{cond:gmax_epsT} The adaptively-set regularization parameter $\gamma_{10}^n$, 
 and the user-set parameters $\GM$ and $\eps_T$ must satisfy the relation
\begin{align}
\gamma_{10}^n \le \GM < \f 1 {\eps_T}.
\end{align}
\end{condition}
The next condition, 
which is the same as Condition (2) of Criteria (4) in \cite{Pollock15b},
gives a necessary criterion for  update of $\gamma_{10}$ 
in order to establish the monotonicity result below.  
\begin{condition}[Condition for the update of $\gamma_{10}$]\label{cond:g10_1}
Given a rate tolerance $\eps_T$ satisfying Condition \ref{cond:gmax_epsT}, 
the ratio of consecutive residual norms must satisfy
\begin{align}\label{cond:gamma10_up}
\left|  \f{\nr{r^{n+1}}}{\nr{r^n}} - \left(1 - \f 1 {\gamma_{10}^n} \right) \right|
        < \eps_T.
\end{align}
\end{condition}
Then, the following result on the monotonicity of the update holds.
This next lemma generalizes the result Corollary 4.7 of \cite{Pollock15b}, 
which establishes the decrease in $\gamma_{10}$, as updated by
\eqref{eqn:update_gamma10} for $\gamma_{10}$ small enough with respect to parameters
$\eps_T$ and $q$.
For practical purposes, however, one may want to start the computation with
a larger value. The following result characterizes
the decrease $\gamma_{10}$ based on the direction cosine of consecutive residuals,
where the direction cosine is given by
\[
\cos(r,s) = \f{\langle r, s \rangle}{\nr{r}\nr{s}}, ~r, s \in \R^n.
\] 
\begin{lemma}[Preasymptotic decrease of $\gamma_{10}$]\label{lemma:gen_gmon}
Given a fixed parameter $0 <q < 1$,
a number $q < \bar q \le 1$, and a rate tolerance $\eps_T>0$ satisfying 
Condition \ref{cond:gmax_epsT}, if $\gamma_{10}^{n+1}$ is 
computed by \eqref{eqn:update_gamma10}, specifically
\begin{align}\label{eqn:compute_gamma10}
\gamma_{10}^{n+1} = 
\max\left\{1 \, , \, 
          q\cdot\f{\nr{r^n}^2}{ \langle r^n, r^n - {r^{n+1]} \rangle}} \right\},
\end{align}
upon satisfaction of
Condition \ref{cond:g10_1}, then 
\begin{align}\label{eqn:gen_gmon001}
\gamma_{10}^{n+1} < \overline q \gamma_{10}^n, ~\oor~ \gamma_{10}^{n+1} = 1,
\end{align}
whenever
\begin{align}\label{eqn:gen_gmon002}
\cos(r^n, r^{n+1}) < \f{\gamma_{10}^n - q/\overline q}{ \gamma_{10}^n(1 + \eps_T) - 1}.
\end{align} 
\end{lemma}
This includes the previous result of 
Corollary 4.7 in \cite{Pollock15b}, as the right-hand side of~\eqref{eqn:gen_gmon002}
satisfies
\begin{align}\label{eqn:gmon_def}
 \f{\gamma_{10}^n - q/\overline q}{ \gamma_{10}^n(1 + \eps_T) - 1}\ge 1,
~\text{ for }~ \gamma_{10}^n \le \gmo(\overline q)
      \coloneqq \f 1 {\eps_T} \left( 1 - \f q {\overline q}\right).
\end{align}
\begin{proof}
Rewriting the update~\eqref{eqn:compute_gamma10} in terms of the direction cosine
\begin{align}
\widehat \gamma_{10} 
 \coloneqq q \cdot \f {\nr{r^n}^2}{{\nr{r^n}}^2 - \langle r^n, r^{n+1} \rangle} 
 = q \cdot \f{1}{1 - \cos(r^n, r^{n+1}) \f{\nr{ r^{n+1} } }{ \nr{r^n} } }
\label{eqn:gen_gmon003}
\end{align}
Applying Condition~\ref{cond:g10_1}, the denominator on the 
right-hand side of ~\eqref{eqn:gen_gmon003} satisfies the inequality
\begin{align}
\f 1 {\gamma_{10}^n} - \eps_T 
   \le 1 - |\cos(r^n,r^{n+1})|\left(1 - \f 1 {\gamma_{10}^n} + \eps_T \right) 
  \nonumber 
  < 1 - \cos(r^n,r^{n+1}) \f{\nr{r^{n+1}}}{ \nr{r^n}},
\end{align}
yielding, for $\cos(r^n, r^{n+1}) > 0$
\begin{align}\label{eqn:gen_gmon004}
\widehat \gamma_{10 }< \f{q}{1 - \cos(r^n,r^{n+1})(1 - 1/\gamma_{10}^n + \eps_T)}. 
\end{align}
Applying \eqref{eqn:gen_gmon004} to bound the desired inequality 
$\widehat \gamma \le \bar q \gamma_{10}^n$, and solving for $\cos(r^n, r^{n+1})$
yields the result \eqref{eqn:gen_gmon001} on satisfaction of~\eqref{eqn:gen_gmon002}.  
For $\cos(r^n,r^{n+1}) \le 0$, the 
result is clear directly from \eqref{eqn:gen_gmon003}.
\end{proof}
Applying the bound from the update condition \eqref{cond:gamma10_up} however 
yields an overly pessimistic view of when the update will decrease, and 
the sufficient condition for decrease of $\gamma_{10}$ given by
\eqref{eqn:gen_gmon002} has been observed in practice to hold only where the 
less general \eqref{eqn:gmon_def} also holds, for $\overline q =1$. 
This is because as the iterations are converging, $\cos(r^n, r^{n+1})$ is generally
close to one. 
The following corollary gives a reliable predictor involving minimal computation,
of when an update will decrease.
\begin{corollary}\label{cor:fast_gamma_decr}
On the hypotheses of Lemma~\ref{lemma:gen_gmon}, namely, given a fixed parameter
$0<q<1$, a number $q < \overline q \le 1$, and a rate tolerance $\eps_T>0$ satisfying
Condition \ref{cond:gmax_epsT}, if $\gamma_{10}^{n+1}$ is computed by
\eqref{eqn:compute_gamma10}, upon satisfaction of Condition \ref{cond:g10_1},
then 
\begin{align}\label{eqn:cor_decr_001}
\gamma_{10}^{n+1} < \overline q \gamma_{10}^n, \oor \gamma_{10}^{n+1} = 1,
\end{align}
whenever
\begin{align}\label{eqn:cor_decr_002}
\breve \eps < \f{1}{\gamma_{10}^n} \left(1 - \f q{\overline q} \right),
\quad ~\text{ for }~ 
\breve \eps \coloneqq 
 \f{\nr{r^{n+1}}}{\nr{r^n}} - \left(1- \f 1 {\gamma_{10}^n} \right). 
\end{align}
\end{corollary}
The proof follows similarly to Lemma~\ref{lemma:gen_gmon}, with $\breve \eps$ 
taking the place of $\eps_T$.
\begin{proof}
From \eqref{eqn:compute_gamma10} and $\breve \eps$ given by \eqref{eqn:cor_decr_002}
\begin{align}\label{eqn:cor_decr_003}
\widehat \gamma_{10} = \f {q}{ 1 - \cos(r^n, r^{n+1})(1 - 1/\gamma_{10}^n + \breve \eps)}.
\end{align}
Solving for $\cos(r^n, r^{n+1})$ to satisfy \eqref{eqn:cor_decr_001} yields the 
equality, {\em c.f., }\eqref{eqn:gen_gmon002}
\begin{align}\label{eqn:cor_decr_004}
\cos(r^n, r^{n+1}) 
= \f{\gamma_{10}^n - q/\overline q }{\gamma_{10}^n(1 + \breve \eps)-1 },
\end{align}
which is assured to hold whenever the right-hand side of~\eqref{eqn:cor_decr_004}
is greater than one, from which the sufficient condition \eqref{eqn:cor_decr_002}, 
for the result ~\eqref{eqn:cor_decr_001}, follows.
\end{proof}
This yields a reliable predictor requiring minimal computation to check if a given update
of $\gamma_{10}$ can be assured to decrease the parameter.  Such a condition can
be enforced if the sequence of parameters is required to decreease monotonically.

\subsection{Picard-like regularization, $\sigma_{01}$} 
\label{subsec:gamma01}
Referring to the inexact iteration given by
\eqref{eqn:iteration_Q} - \eqref{eqn:iteration_rn},
a strict Newton-like iteration prescribes $\gamma_{01} = 1 = \gamma_{10}$,
{\em i.e.,} $\sigma_{10} = 0$
while a strict Picard-like iteration prescribes $\gamma_{01} = 1$ and
$\gamma_{10} = 0$.  From the current generalized standpoint, $\sigma_{01}$
controls the additional diffusion-like term $A(u^n;w^n)$, which is 
proposed here to balance the linearization error 
$\lin(u^n)$, given by ~\eqref{eqn:comperr_004},
stabilizing the iteration by adding diffusion to the system.  

As seen
in the residual representation~\eqref{eqn:comperr_003}, the computation
of the linearization error up to the contribution from the floating-point
error $\fl$,  can be accomplished by subtracting the remaining
terms to the other side of the equation.
This is viewed as preferable to the computation of $\lin(u^n)$ directly
by the definition \eqref{eqn:comperr_004}, as it requires only 
matrix-vector multiplications, and in particular does not require the assembly
of the term $A_1'(u^n;u^{n+1})\w^n$.

Minimizing the $L_2$ of the sum of the Picard-like regularization and linearization
error based on the latest information, to determine a new value for
$\sigma_{01}$ yields
\begin{align}\label{eqn:sigma001}
\tilde \sigma &= \argmin \nr{\sigma A(u^{n+1};w^n) + \lin(u^n)}
\\ \nonumber
&= -\f{\langle \lin(u^n), \,
              A(u^{n+1};w^{n}) \rangle}{\nr{A(u^{n+1};w^{n})}^2}.
\end{align} 
Up to pollution by floating-point error $\fl$, and allowing only positive
contributions from $\sigma_{10}$, {\em i.e.,} applying this regularization 
only when it adds to, not subtracts from, diffusion to the system
\begin{align}\label{eqn:sigma002}
\sigma_{01}^{n+1}  & = \max\left\{0\, , \,  
   \f{\left\langle -r^{n+1} + (1-\f{1}{\gamma_{10}})r^{n} 
   +\f{\alpha^{n}}{\gamma_{10}}
    R\w^{n} + \sigma_{01}^{n}A(u^{n+1};w^{n}), \,
                            A(u^{n+1};w^{n}) \right\rangle}
                    {\nr{A(u^{n+1};w^{n})}^2} \right\}.
\end{align}

Then the Picard-like regularization is given by
\begin{align}\label{eqn:update_gamma01}
\gamma_{01}^n & = \gamma_{10}(1 + \sigma_{01}^n). 
\end{align}
This definition given by~\eqref{eqn:sigma002}-\eqref{eqn:update_gamma01} 
uses the information from the latest iterate to add 
problem-dependent diffusion to have the greatest cancellation effect
on the linearization error.

\subsection{Tikhonov-like regularization, $\alpha^n$}
\label{subsec:alpha}
The parameter $\alpha^n$ which scales the regularization term $\phi(w,v)$ 
is seen to come from the inverse of the pseduo-time
step, as in \eqref{eqn:reg_004}.  However, as discussed in \cite{Pollock14a,Pollock15a},
this parameter is also seen to scale the analogous regularization term 
found by applying a Tikhonov-type regularization to the linearized system
\cite{Engl1996}. 
In the current numerical experiments, as well as previous ones by the author,
the regularization term 
$\phi(w,v) \coloneqq \int_\Omega \beta(x,u) \grad w \cdot \grad v$, where $\beta$
may be as simple as the identity, or it may be a cutoff function computed once on 
each mesh refinement, or a function of $u^0$, the initial iterate on each refinement.
While $\beta$ could be chosen as function of $u$ and updated on each iteration, this 
would increase assembly costs which can already be high compared to the solve-time
in the preasymptotic regime. The experiments in Section \ref{sec:numerics} show
$\beta = 1 + |\kappa'(u_k^0)|$, on refinement $k$, 
in the first example; and, $\beta = 1$ in the second example.  
The is as opposed to the Picard-like regularization 
controlled by $\sigma_{01}^n$ which
adds diffusion scaled by the latest $\kappa(u^n)$ on each iteration.  
The proposed parameter is scaled by
$\gamma_{10}/\nr{R w^n}$, and is guided by the contribution of the remaining
second order terms, so long as it does
not interfere with the convergence rate. 
\begin{align}\label{eqn:alpha0}
\alpha^0 &= \nr{r^0}, 
\\ \label{eqn:update_alpha}
\alpha^{n+1} & = \f{\gamma_{10}}{\nr{R\w^{n}}} \cdot 
\min\left\{\left\| r^{n+1} - \left( 1-\f{1}{\gamma_{10}} \right) r^n 
 - \f{1}{\gamma_{10}} \alpha^nR\w^n \right\|, \, \f{\eps_T}{2}\nr{r^{n+1}}  \right\},
\quad n \ge 1.
\end{align}
This is in contrast to the scaling of the Tikhonov term proposed in 
\cite{Pollock14a,Pollock15a,Pollock15b}, which is guided by the norm of the
residual.  The new definition chooses a generally smaller term for $\alpha^n$,
so as not to interfere with the convergence rate farther from the solution,
{\em i.e., } in the preasymptotic and coarse mesh regimes.
The plots of the terminal $\alpha$ and $\alpha \nr{R\w}$ on each refinement
for the examples in Section \ref{sec:numerics} highlight the importance of 
normalizing $\alpha$ against $\nr{R\w^n}$ to control the contribution from this
regularization.
It is noted as well that applying~\eqref{eqn:alpha0} and \eqref{eqn:update_alpha}
selects a larger regularization parameter on the first iteration of each refinement,
which stabilizes the correction from the interpolation of the previous solution onto the
finer mesh.

\subsection{Inexact scaling regularization, $\delta$}
\label{subsec:delta} 
The scaling parameter $\delta$ is adjusted after the last iteration 
on each refinement.  After the iteration completes on refinement $k$,
let the final iterate $u_k$ be indexed by $u^{n+1}$.  Then rearranging
terms in~\eqref{eqn:iteration_mainQ}, the residual $r^n$ satisfies
\begin{align}\label{eqn:updatedelta_001}
-r^n =   \alpha^nR\w^n + \gamma_{10}(A_1'(u^n;u^n) + A_2'(u^n))\w^n
    +\sigma^n\gamma_{10}A(u^n;w^n). 
\end{align}
Define now a new quantity $r^\cL$ to satisfy
\begin{align}\label{eqn:updatedelta_001a}
-r^\cL =  \alpha^nR\w^n + \gamma_{10}(A(u^{n+1}; u^{n+1}) - A(u^n; u^n))
    +\sigma^n\gamma_{10}A(u^n;w^n) + A(u^n; u^n), 
\end{align}
where the Jacobian terms of \eqref{eqn:updatedelta_001} have been replaced by the
difference $A(u^{n+1};u^{n+1}) - A(u^n; u^n)$, which they approximate.
Then the difference between \eqref{eqn:updatedelta_001a} 
and \eqref{eqn:updatedelta_001}, is given by
\begin{align}\label{eqn:updatedelta_001b}
-(r^{\cL} - r^n) = \gamma_{10} \, \lin(u^n).
\end{align}
The number $\widetilde \delta$ is now set so that $\widetilde \delta - A(u^n; u^n)$ 
approximates $r^\cL$.  In particular
\begin{align}\label{eqn:updatedelta_001c}
\widetilde \delta = \argmin \nr{ \widetilde \delta f_Q - A(u^n; u^n)  - r^\cL}
          = \argmin\nr{ (\widetilde \delta - \delta_k)f_Q + \gamma_{10} \, \lin(u^n)},
\end{align}
where the last equality follows by \eqref{eqn:updatedelta_001b} and 
\eqref{eqn:iteration_rn}.
Up to floating point error, this may be computed by
\begin{align}\label{eqn:updatedelta_003}
\tilde \delta = \f{\langle f_Q, 
          \alpha^n R\w^n + \gamma_{10}(A(u^{n+1};u^{n+1}) - A(u^n; u^n))
         +\sigma^n \gamma_{10}A(u^{n}, w^n) + A(u^n; u^n) \rangle  }
         {\nr{f_Q}^2}.
\end{align}
Then, for a user set parameter $0 < q_{k} < 1$ as in \eqref{eqn:update_gamma10}, 
$\delta_{k+1}$ is set by
\begin{align}\label{eqn:updatedelta_004}
\delta_{k+1} = \max\left\{ \f 1 {q_k} \tilde \delta \, , \,  1 \right\}.
\end{align}
This update then satisfies the property that $\delta_k = 1$ for all $k\ge K$ for 
some finite $K$, as shown in~\cite{Pollock15b}.
In particular, \eqref{eqn:updatedelta_001c}, demonstrates that
$\widetilde \delta = \delta_{k}$, if the linearization at $u^n$ is exact.  
In this case $\delta_{k+1} = q_k^{-1}\delta_k$.  The adjustment in $\delta$
is seen to systematically reduce the residual regularization $S^{res}$, 
while maintaining sensitivity to the component of the linearization error
along the direction of the source.

The parameter $q_k$ may be taken as the constant $q$ used in 
\eqref{eqn:compute_gamma10}, the computation of $\gamma_{10}$, or may be updated
to increase the parameter $\delta$ more aggressively when $\gamma_{10}$ has
been updated more often or is sufficiently close to one.  For instance
\begin{align}\label{eqn:qfordelta}
q_k = \min\{q^P, q^{(1 + 1/\gamma_{10})} \},
\end{align} 
where $P$ is the number of updates of $\gamma_{10}$ on refinement $k$.
\subsection{Regularized adaptive algorithm}
\label{subsec:adaptive_algorithm} 
The regularized adaptive algorithm
of~\cite{Pollock15b} effectively traversed the coarse mesh and preasymptotic
regimes starting from a coarse mesh where the solution-dependent coefficients 
of~\eqref{eqn:weak} were unresolved. The described method was demonstrated to
uncover the internal layers and arrive at the asymptotic phase of Newton iterations 
for the well-resolved problem.
The method was not uniformly efficient, however, especially for large values
of the numerical dissipation parameter.  In particular, the method was
allowed to continue to iterate while converging at the predicted rate until
the norm of the residual dropped sufficiently below the level of the 
previous residual. This resulted in longer computational times in the
coarse mesh regime, where the linear problems to solve are 
significantly smaller than in the asymptotic regime, but the linear convergence
rate may be very slow.  
To remedy this situation, the current algorithm is accelerated by splitting 
it into three phases, roughly corresponding to the coarse mesh, 
preasymptotic and asymptotic phases.  

A modified strategy to update the numerical dissipation and 
exit the iterations in the first
phase allows significantly faster progression to the
second phase, roughly correlated to the preasymptotic regime
A modified parameter-update in the
second phase leads to faster progression to third phase, correlated to the
asymptotic regime; and, prevents stalling of the final parameter updates due
to pollution from the floating-point error.

The first phase, $\gamma_{10} > \gmo$, exits early upon update of $\gamma_{10}$, 
and does not require residual reduction. 
The second phase, $\gmo \ge \gamma_{10} > 1$,
corresponds to the preasymptotic phase.  The updates of $\gamma_{10}$ are
guaranteed to be monotonically decreasing in this phase as developed in 
\cite{Pollock15b}.  The third or final
stage corresponds to the asymptotic phase of the algorithm where
the iterations converge quadratically as is standard for Newton methods once
the solution iterate has entered the basin of convergence.

\subsubsection{User-set parameters}\label{subsubsec:userset}
The following user-set parameters are used in the adaptive algorithm. The first 
three, $\GM,$ the maximum value of $\gamma_{10}$; $\eps_T$, the rate-tolerance; and
$q,$ the reduction factor used in the parameter update ~\eqref{eqn:update_gamma10} 
for the update of $\gamma_{10},$ 
can be consolidated into
two.  From  \eqref{cond:gamma10_up}, and as discussed in \cite{Pollock15b}, 
convergence of the residual under this
regularization structure requires $\eps_T < 1/\GM$, so it is natural to define
$\eps_T = q/\GM$. Then setting $0 < q < 1$ and $\GM > 1$ defines both $\eps_T$ and 
$\gmo = \GM(1/q -1)$, from \eqref{eqn:gmon_def}, for which the update guarantees
monotonic decrease of $\gamma_{10}$ whenever $\gamma_{10} < \gmo$, in accordance
with Lemma \ref{lemma:gen_gmon}. 

The two remaining parameters are standard for iterative methods. The tolerance for
the residual denoted $\tol$ used to exit iterations in the asymptotic regime is 
set to a constant value: $\tol = 10^{-7}$ in the current results.  The maximum
number of iterations allowed $\itmax$ is set to a default of 20 iterations, 
but modified in the preasymptotic phase where $\gamma_{10} > 1$, 
to allow residual reduction with iterations converging at the 
accepted rate $1 - 1/(2 \, \gamma_{10})$, by
\begin{align}
\itmax = 1+\ceil(\log\nr{r_{k-1}} - \log\nr{r^0})/\log((1-1/(2\, \gamma_{10})),
\end{align}
where $r_{k-1}$ denotes the terminal residual on refinement $k-1$,  
$r^0$ denotes the initial residual on refinement $k$, and
$\ceil(\cdot)$ denotes the ceiling function.

\subsubsection{Incomplete solves: exiting the iterations}\label{subsubsec:exit}
One of the key features of the regularized method is the early-exit of the 
nonlinear iterations in the pre-asymptotic and coarse mesh regimes.  The exit
is governed by observed convergence of the iterations at the predicted rate.  
This shows the linearization error is low enough to extrapolate stability of 
the approximate Jacobian at the current iterate.  As such, the interpolation of 
that iterate onto the refined mesh is suitable for a starting guess fo the next
nonlinear solve.   In contrast to the methods of \cite{Pollock15a,Pollock15b}, the
first exit-condition allows exit from the iterations on the conditions for 
updating $\gamma_{10}$, in the initial phase where $\gamma_{10}> \gmo$.
In particular, residual reduction is only enforced once $\gamma_{10}$ is small
enough to assure decrease towards one on every update.
\begin{criteria}[Exit criteria]\label{criteria:exit}
Let $\beta^{n+1} = \nr{r^{n+1}}/\nr{r^n}$.
The nonlinear iterations are terminated on iteration $n$ of level $k$ on 
satisfaction of one of the following 
sets of conditions.
\begin{enumerate}
\item The first set of conditions accelerates the algorithm through the coarse mesh
regime. 
\begin{align}
 \gamma_{10} &> \gmo, \label{exit:cm1}\\
|\beta^n - \beta^{n-1}|  &\le \eps_T, \label{exit:cm2}\\
\left| \beta^n - \left(1 - \f 1 {\gamma_{10}} \right)  \right| &< \eps_T, 
\label{exit:cm3} \\
n &> 2 \label{exit:cm4}.
\end{align}
\item The second set of conditions feature sufficient reduction 
of the residual, a relaxed convergence rate tolerance, and a stability criterion.
This set of conditions allows for successful exit of the iterations in the preasymptotic 
regime. 
\begin{align}
\nr{r^{n+1}} &< \nr{r^n}, \label{exit:pa1}\\
\nr{r^n}     &\le \min \{ \nr{r^0},\nr{r_{k-1}} \}, \label{exit:pa2}\\
\beta        &< \left(1 - \f 1 {2 \gamma_{10}} \right), \label{exit:pa3}\\
| \beta^{n-1} - \beta^n| &\le \f{\eps_T}{2}. \label{exit:pa4} 
\end{align}
\item The third condition allows successful exit in the asymptotic regime: the 
iterations have converged to tolerance.
\begin{align}
\nr{r^{n+1}} \le \tol. \label{exit:ar}
\end{align}
\item The fourth exit condition detects failure of the iterations to converge:
either sufficient increase of the residual, or reaching the maximum number 
of iterations.
\begin{align}
\beta > 1 + \f 1 {\gamma_{10}}, ~\oor~ n > \itmax. \label{exit:fail}
\end{align}
\end{enumerate}
\end{criteria}

\subsubsection{Regularization update conditions}
In agreement with \eqref{cond:gamma10_up}, 
the numerical dissipation parameter $\gamma_{10}$ is
updated on satisfaction of
\begin{align}\label{condition:update_gamma}
\gamma_{10}^n > 1, \quad \eqref{exit:cm2}-\eqref{exit:cm3}, 
\an \gamma_{10}^n = \gamma_{10}^{n-2}.
\end{align} 
The last condition requires the parameter $\gamma_{10}$ is updated at most
every three iterations, allowing the iterations to stabilize and the comparison
of convergence rates over three iterations to be meaningful.

The scaling parameter $\delta$ is updated after the terminal iteration on 
satisfaction of 
\begin{align}\label{condition:update_delta}
\delta_k < 1, \an \text{Conditions (1), (2) or (3) of Criteria~\ref{criteria:exit}}. 
\end{align}
The parameter updates, exit criteria and adaptive mesh refinement are 
summarized in the following regularized adaptive algorithm.
\begin{algorithm}
[Algorithm using the inexact iteration~\eqref{eqn:iteration_Q}]
\label{algorithm:basic}
Set the parameters $q$ and $\GM$. 
Start with initial $u^0$, $\gamma^0$, $\delta_0$ and $\sigma_0^0 = 0$.
On partition $\cT_k, ~ k = 0, 1, 2, \ldots$
\begin{enumerateX}
\item Compute regularization matrix $R_k$.
\item Set $r^0 = \delta f_Q-A(u^0;u^0)$, and set $\alpha_0 = \nr{r^0}$.
      Set $\sigma_k^0 = \sigma_{k-1}, ~k \ge 1$.
\item While Exit criteria~\ref{criteria:exit} are not met on iteration $n-1:$
\begin{enumerateY}
  \item Solve \eqref{eqn:iteration_Q} for $\w^n$.
  \item Update $u^{n+1} = u^n + w^n$, 
        and $r^{n+1} = \delta_Q f_k-A(u^{n+1};u^{n+1})$.
  \item If the conditions~\ref{condition:update_gamma} are satisfied, 
        update $\gamma_{10}^{n+1}$ by \eqref{eqn:compute_gamma10}.
  \item Set $\sigma_{01}^n$ by to~\eqref{eqn:sigma002},
        and set $\alpha^n$ by to \eqref{eqn:update_alpha}.
\end{enumerateY}
  \item If Condition~\ref{condition:update_delta} is satisfied,
        update $\delta_{k+1}$ for partition $\cT_{k+1}$ according
        to \eqref{eqn:updatedelta_003}-\eqref{eqn:updatedelta_004}, 
        with $q_k$ set by \eqref{eqn:qfordelta}. 
  \item Compute the error indicators to determine the next
        mesh refinement.
\end{enumerateX}

\end{algorithm}

The numerical results in the following section are computed using standard 
{\em a posteriori} residual-based error indicators, 
as in for instance~\cite{Stevenson07}.  For the nonlinear anisotropic 
problem~\eqref{eqn:anisotropic}
the local indicator for element $T \in \cT_k$ with $h_T$ the element diameter is 
given by
\begin{align}
\label{eqn:jump_indicators}
\zeta_T^2(v) &= \zeta_{\cT_k}^2(v,T) \coloneqq h_T \nr{ J_T(v)  }_{L_2(\pa T)}^2 \\
\label{eqn:indicators}
\eta_T^2(v) &= \eta_{\cT_k}^2(v,T) \coloneqq h_T^2 
\left\| 
\sum_{i,j = 1}^n \f{\pa}{\pa x_j} \left( \kappa_{ij}(v) \f{\pa }{\pa x_j}v\right) 
                 +f \right\|_{L_2(T)}^2 
+  \zeta_T^2(v),
\end{align}
$J_T(v) \coloneqq \llbracket \sum_{i,j=1}^n\kappa_{ij}(v) \f{\pa}{\pa x_j} v  \cdot n_i 
                  \rrbracket_{\pa T}$,
with jump
$\llbracket \phi \rrbracket_{\pa T} 
\coloneqq {\lim_{t \goto 0} \phi(x + t n) - \phi(x - tn)}$,
 where  $n= (n_1, n_2)$ is the appropriate outward normal defined on $\pa T$.
The error estimator on partition $\cT_k$ is given by the $l_2$ sum of
indicators $\eta_{\cT_k}^2 = \sum_{T \in \cT_k}\eta_{T}^2$.

\section{Numerical results}\label{sec:numerics}
Two numerical examples illustrate instances where different terms in the
regularization are active to stabilize the iterations. The first example
demonstrates Algorithm \ref{algorithm:basic} on a problem where $\kappa'(s)$ is Lipschitz
but not uniformly differentiable.  In particular, it has a corner at $s = 1/2$.
This model problem shows an anisotropic shift in solution-dependent diffusion, and
features steep gradients in the diffusion coefficient. The approximate Jacobian
requires continued regularization from both the Picard-like term controlled by 
$\sigma_{01}$, and the Tikhonov-like term, scaled by $\alpha$.
In the second example, the regularization is driven by $\gamma_{10}$ and $\delta$
in the preasymptotic phase while a thin internal layer in the diffusion is 
uncovered.  For this problem, the regularization has diminished importance in 
the asymptotic regime, whereas computing a sequence of iterates that enter 
the asymptotic regime strongly depends on the regularization.

The simulations were performed with a Python implementation
of the FEniCS library~\cite{LMW12a}, with the parameter computations, specifically
where sparse matrix-vector products are necessary, computed with the PETSc 
backends \cite{petsc-user-ref}. Running the simulations on an 4GHz Intel Core i7 
iMac, the first example runs to full residual convergence in less than 150 sec., on 
refinement level 33;  and the second in under 20 sec., on refinement level 36.  

In both examples the initial iterate $u_0^0 = 0$, 
and thereafter $u_k^0$ is interpolated from $u_{k-1}$, the terminal iterate on
refinement $k-1$, onto the refined mesh partition of level $k$. 
The initial scaling parameter $\delta = 1/\GM$, and $\GM$ is specified in each 
example.  Both examples use the regularization reduction factor $q = 0.865$.
Each simulation is discretized with linear Lagrange basis elements and was started with 
a uniform initial mesh partition of 144 elements.

\begin{example}[Anisotropic diffusion]\label{example1} Consider the 
anisotropic diffusion equation on $\Omega = (0,1) \times (0,1)$
\begin{align}
-\div (\kappa(u) \grad u) = f(x,y) \inn \Omega, ~ u = 0 ~\on~ \pa \Omega,
\end{align}
with 
\begin{align}
\kappa(u) &= \left( \begin{array}{cc}\kappa_{11}(u)& 0 \\
                                    0             & \kappa_{22}(u) \end{array}\right),
\nonumber \\
\kappa_{jj}(u) & = k + \tanh((1/\eps_j)(u-a)^2\,\sign(u-a)), ~j = 1,2,
\end{align}
with the parameters $a = 0.5$, $\eps_1 = 4\times 10^{-4}$, $\eps_2 = 4 \times 10^{-2}$,
and $k = 2$.
The discontinuity in $\kappa'$ at $u = 1/2$ separates this problem from the 
classes of Remark~\ref{remark:known_classes}
where asymptotic convergence is known, assuming a fine enough mesh.
The source function $f$ is given by
\begin{align}
f(x,y) & = 2(1-x)(1-y)(e^{6x^2}-1)(e^{6y^2}-1).
\end{align}
The initial regularization parameter $\gamma_{10}$ is set as $\GM = 5$ and the 
regularization function $\phi(w,v) = (1+\kappa'(u_k^0) \grad w, \grad v)$.
\end{example}
\begin{figure}
\centering
\includegraphics[trim=120pt 80pt 100pt 110pt,
clip=true,width=0.30\textwidth]{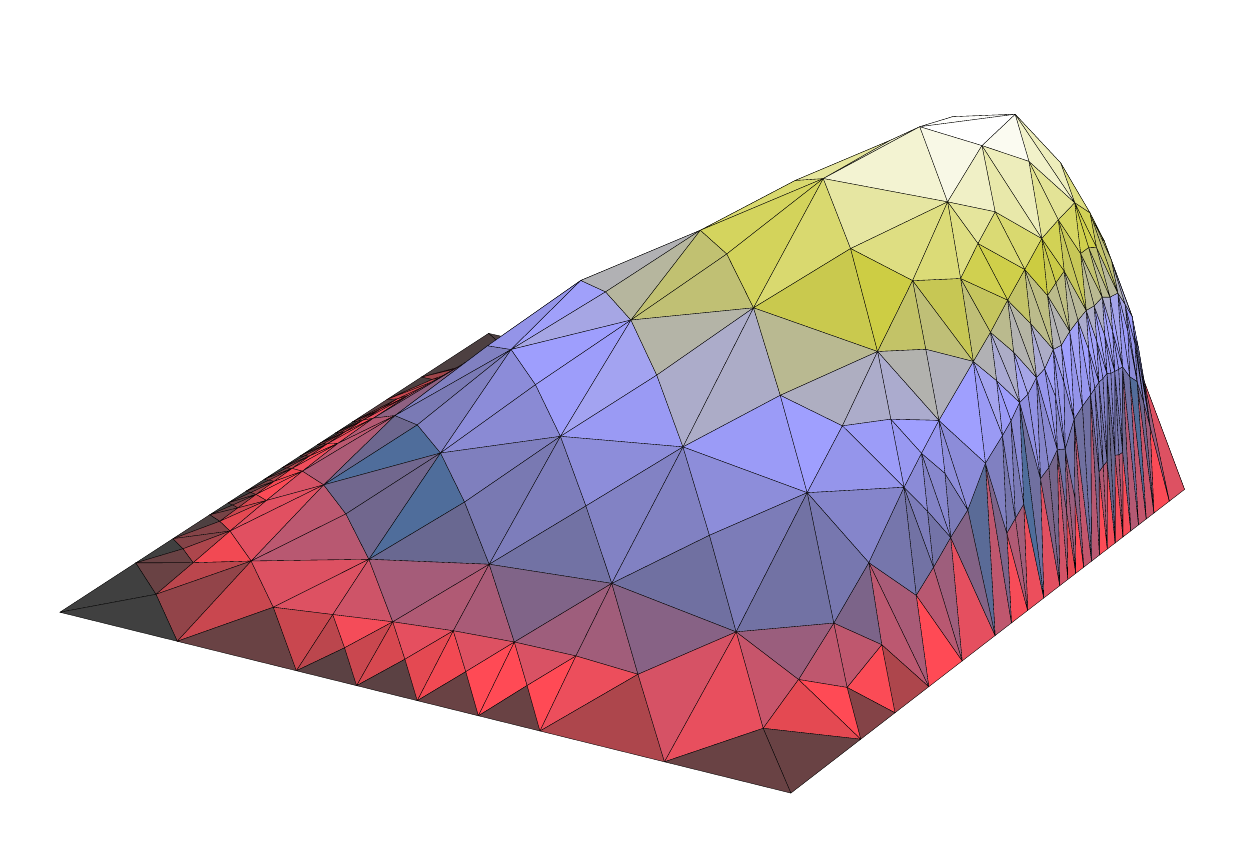}~\hfil~
\includegraphics[trim=120pt 80pt 100pt 110pt,
clip=true,width=0.30\textwidth]{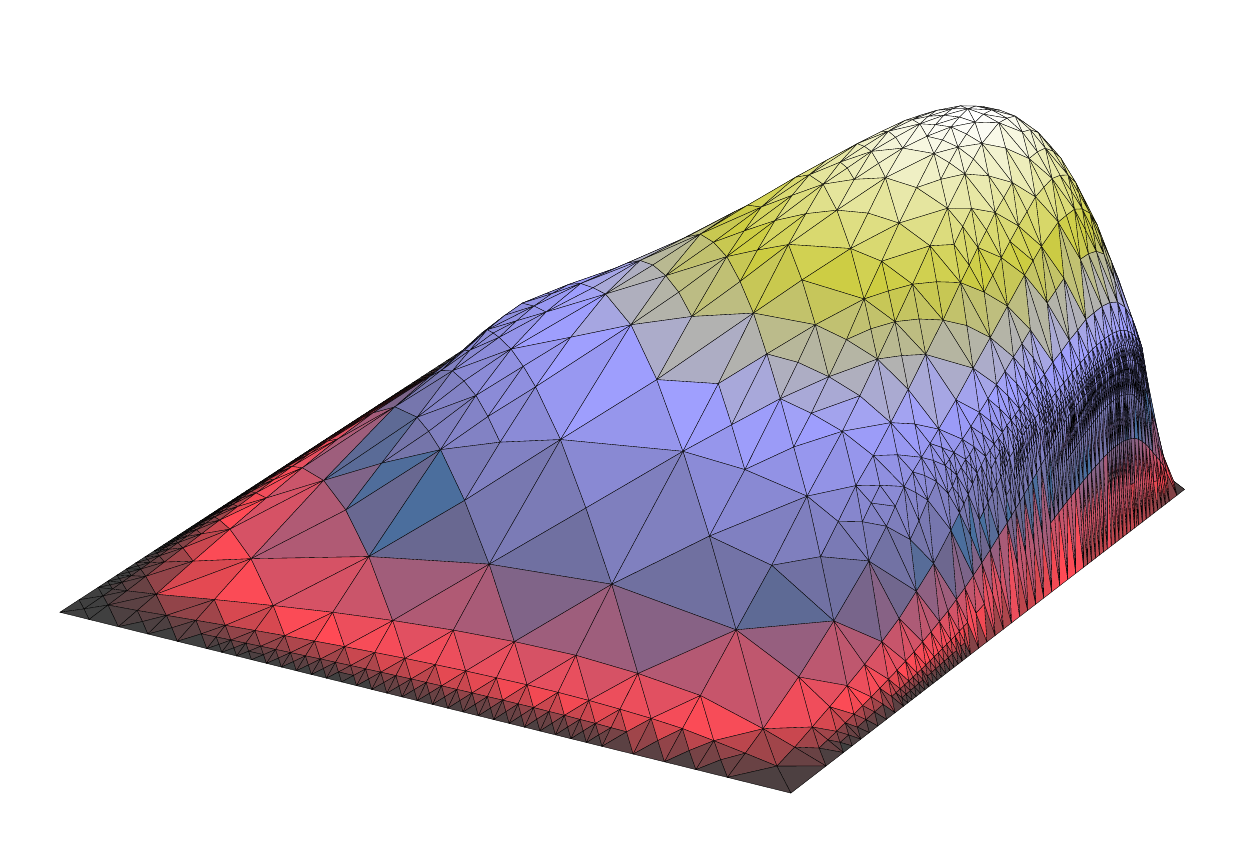}~\hfil~
\includegraphics[trim=120pt 80pt 100pt 110pt,
clip=true,width=0.30\textwidth]{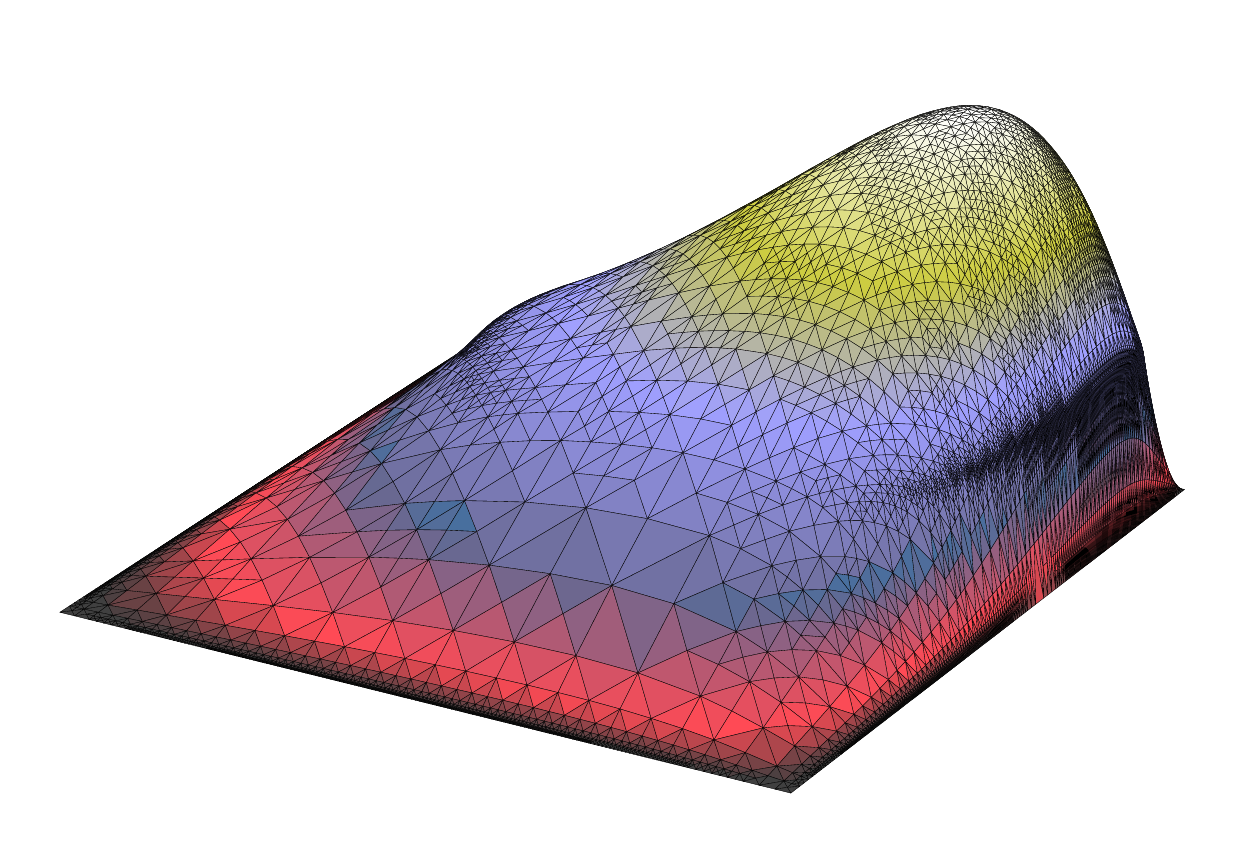}~\hfil~
\caption{Terminal solution iterates from Example~\ref{example1}.
Left: solution iterate with $\gamma_{10}= 5$ on level 10 with 237 dof.  
Center: solution iterate with $\gamma_{10}=3$ on level 20 with 1332 dof.
Right: solution iterate with $\gamma_{10} = 1$ on level 30 with 9613 dof. 
}
\label{fig:ex1sol}
\end{figure}
\begin{figure}
\centering
\includegraphics[trim=40pt 40pt 40pt 40pt,
clip=true,width=0.32\textwidth]{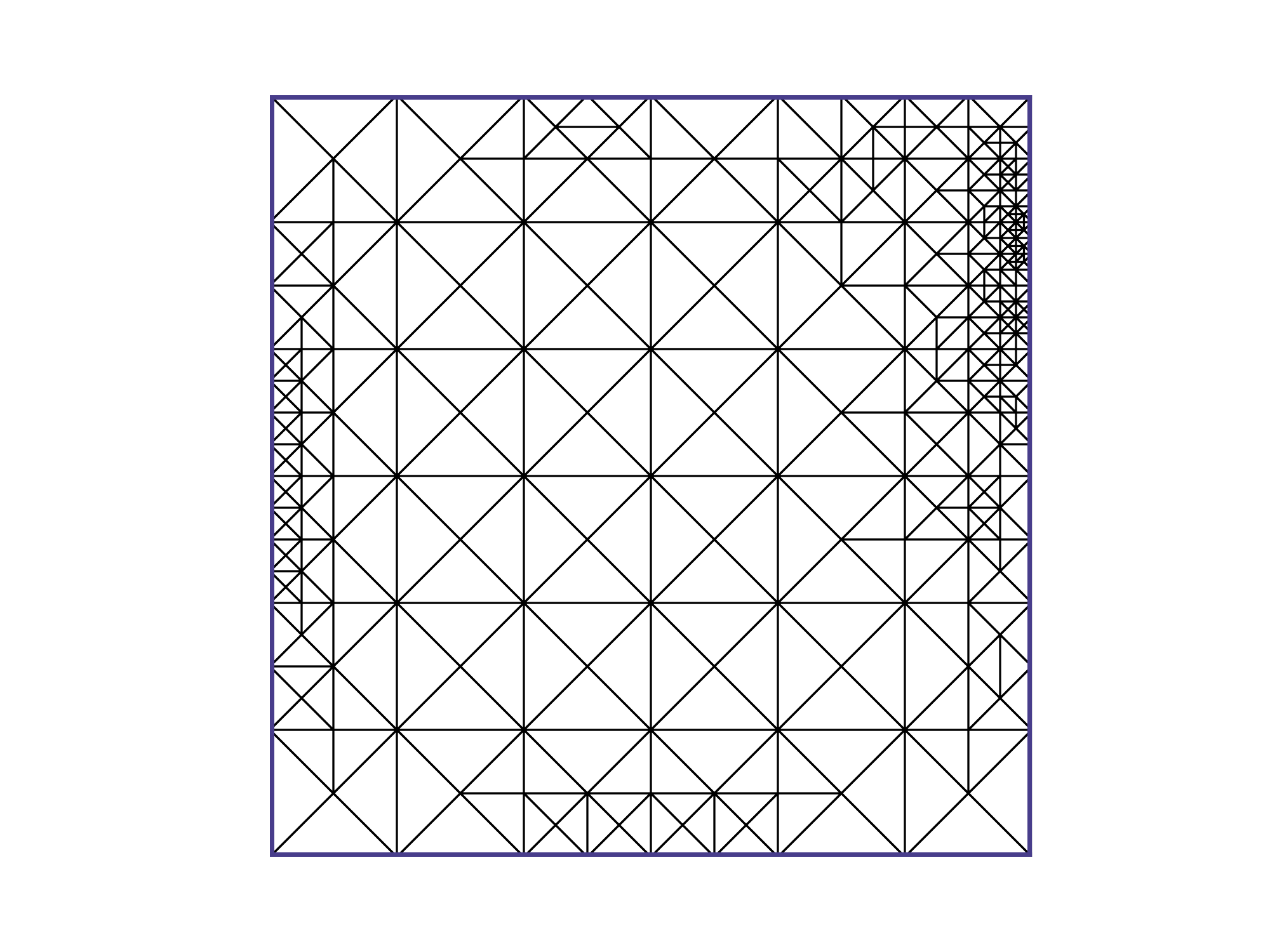}~\hfil~
\includegraphics[trim=40pt 40pt 40pt 40pt,
clip=true,width=0.32\textwidth]{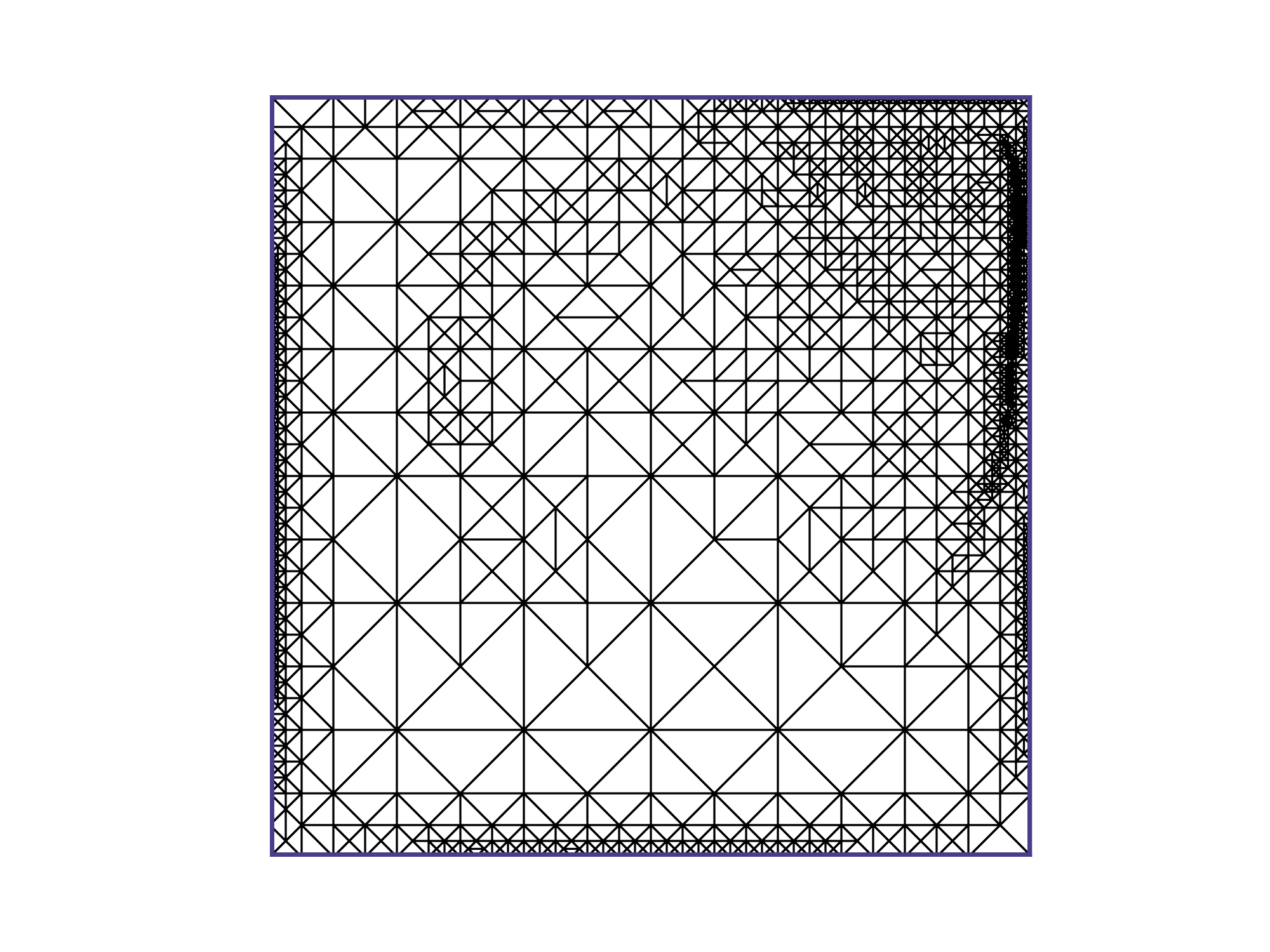}~\hfil~
\includegraphics[trim=40pt 40pt 40pt 40pt,
clip=true,width=0.32\textwidth]{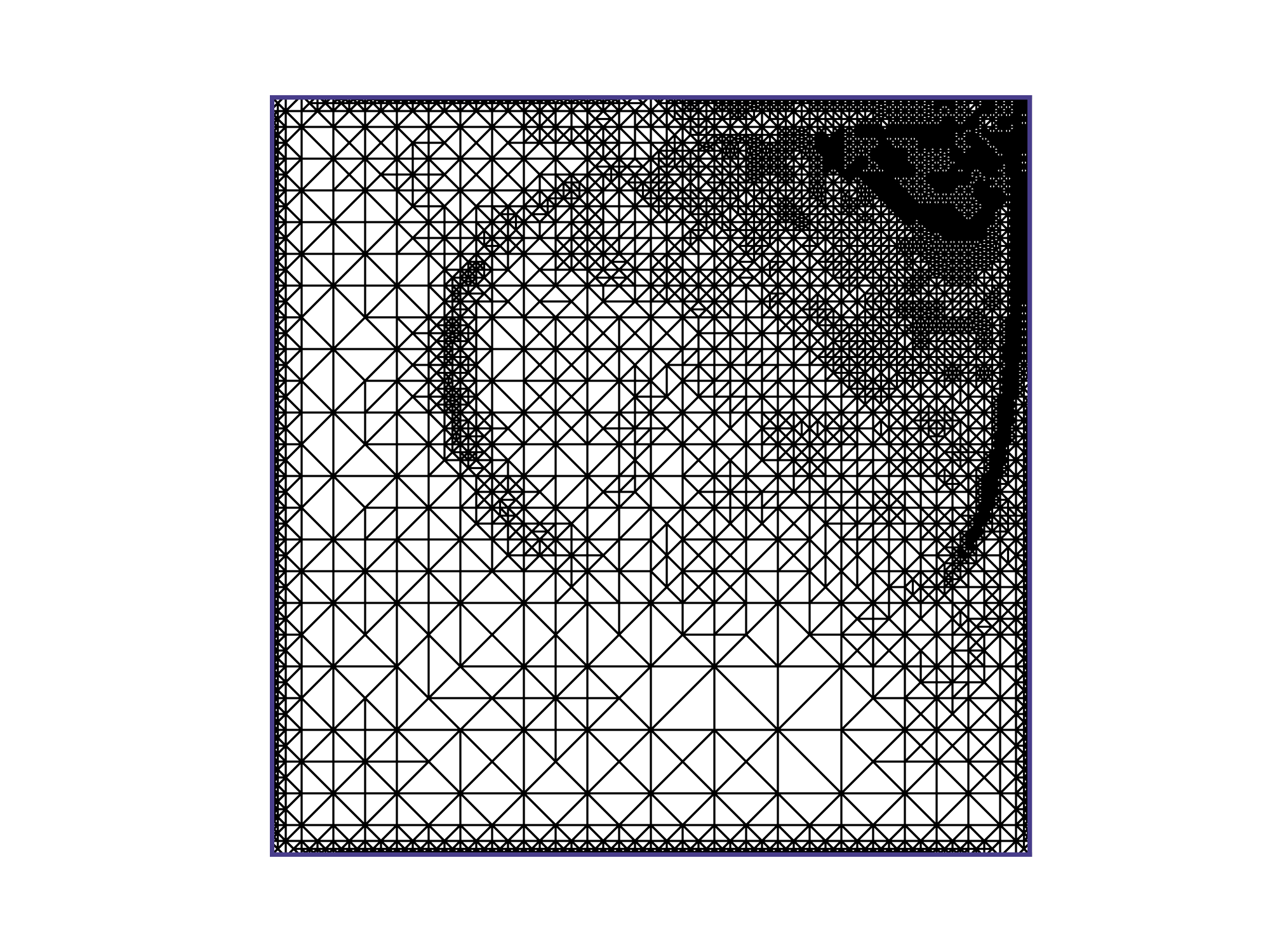}~\hfil~
\caption{Adpative meshes from Example~\ref{example1}.
Left: mesh on adaptive level 10 with 237 dof.  
Center: mesh on adaptive level 20 with 1332 dof.
Right: mesh on adaptive level 30 with 9613 dof. 
}
\label{fig:ex1mesh}
\end{figure}
\begin{figure}
\centering
\includegraphics[trim=0pt 0pt 0pt 0pt, clip=true,width=0.45\textwidth]
{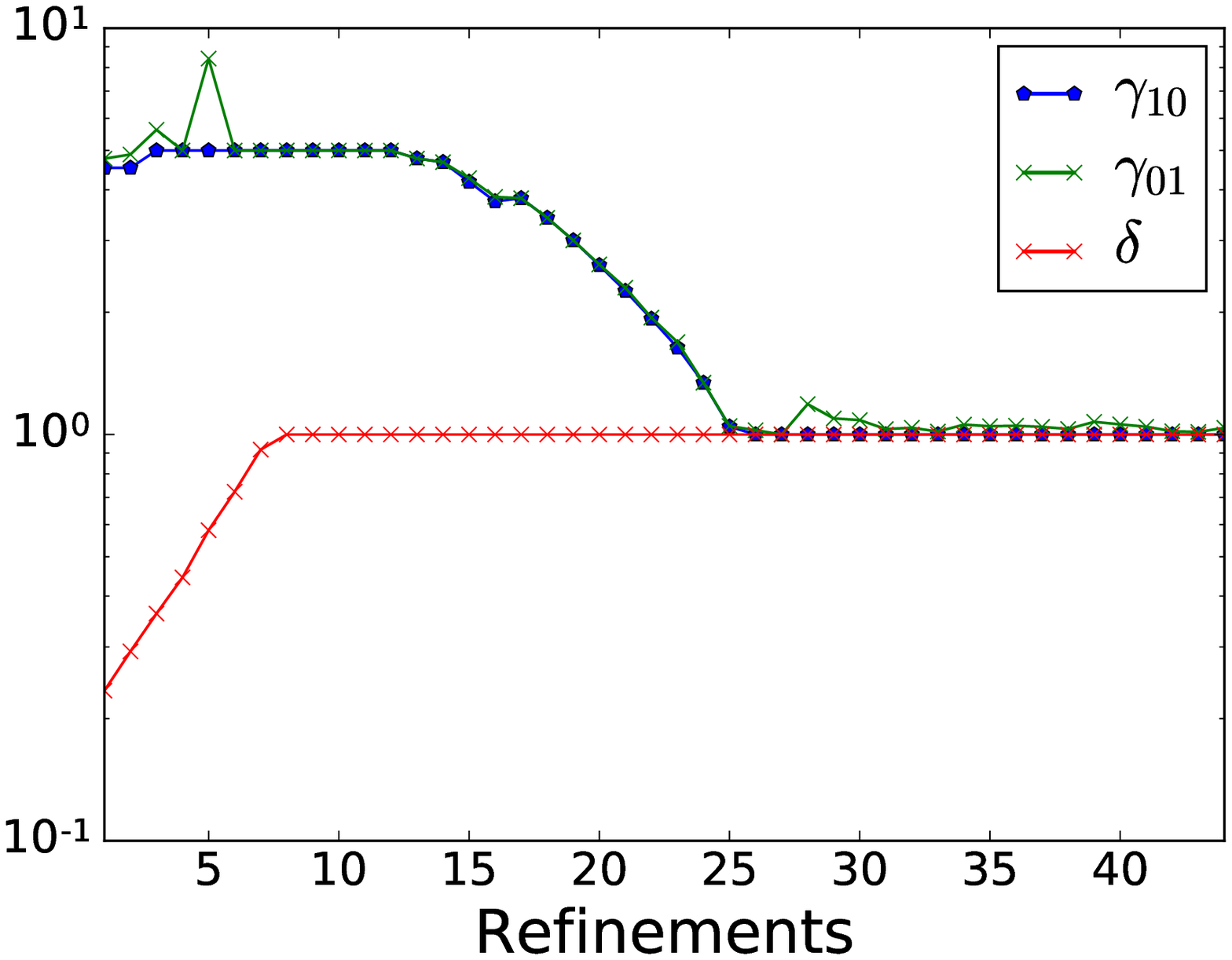}
\includegraphics[trim=0pt 0pt 0pt 0pt, clip=true,width=0.45\textwidth ]
{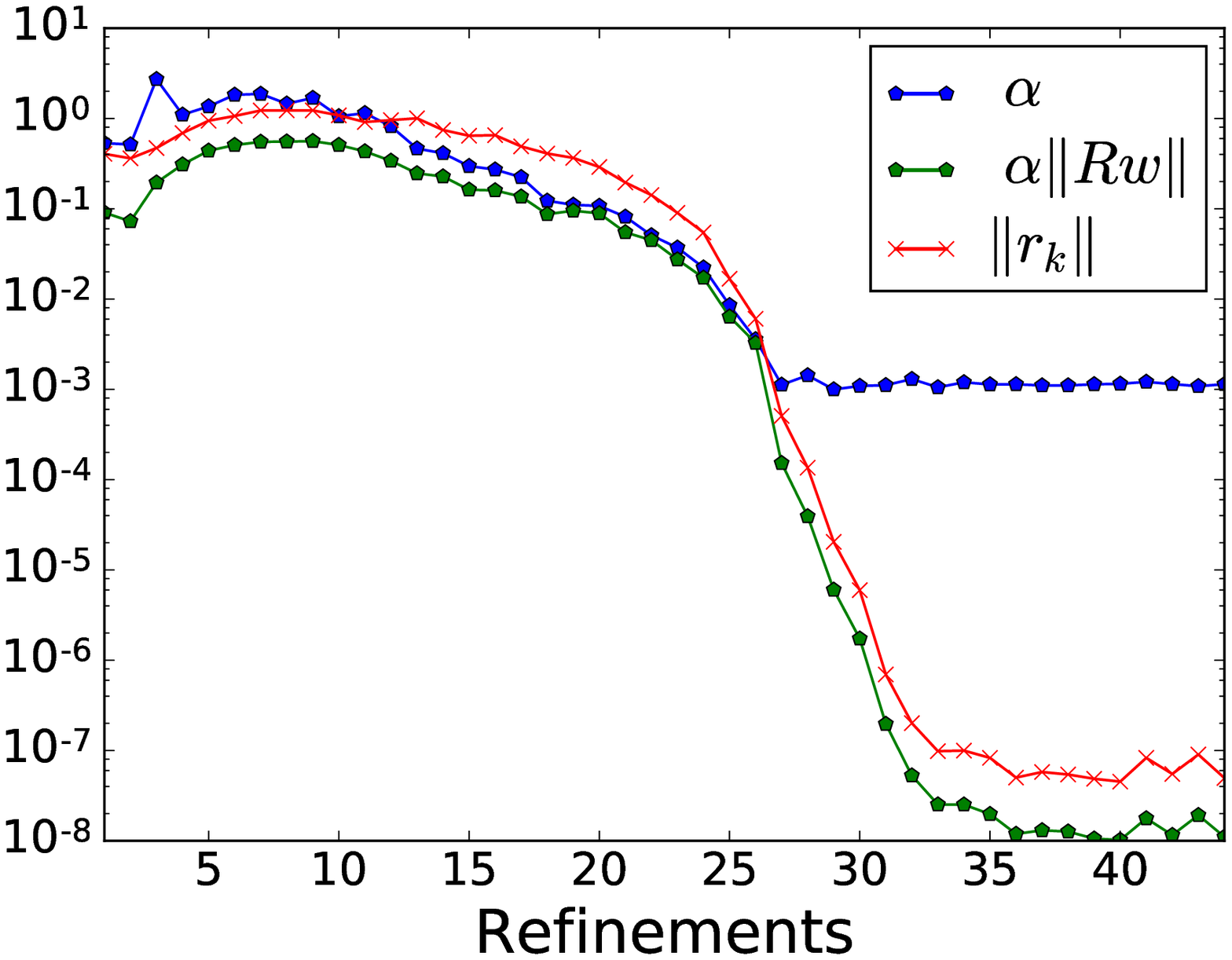}~
\caption{
Left: regularization parameters $\gamma_{10}$, $\gamma_{01}$  and $\delta$.
Right: regularization parameter $\alpha$, norm of regularization $\alpha\nr{R\w}$, 
       and the norm of the terminal residual
       $\nr{r_k},$ for Example~\ref{example1}, with nonlinear iterations running 
       to tolerance $\tol = 10^{-7}$.}
\label{fig:ex1Rparam}
\end{figure}
Figure~\ref{fig:ex1sol} shows snapshots of the computed iterates on refinements 
$10, 20$ and $30$, with respectively $237, 1332$ and $9313,$ degrees of freedom (dof),
illustrating the progress from the preasymptotic
into the asymptotic regimes. Figure~\ref{fig:ex1mesh} shows the corresponding 
adaptive meshes.  The solution plots and meshes illustrate how the mesh is refined
for both the boundary layer on either side of the origin; and, for the steep gradients
in the diffusion coefficient.  In this anisotropic case, the gradients are 
orders of magnitude steeper in the $x$- direction than the $y$- direction.  
In particular, $\kappa_{11}'$ has a Lipschitz constant of approximately $2500$, while
$\kappa_{22}'$ has a Lipschitz constant of approximately $25$.  It is observed
that the meshes refine more in the vicinity of the steeper gradients;
and, the mesh partition remains relatively coarse over large areas of the domain.

Figure~\ref{fig:ex1Rparam} shows the terminal value of the regularization parameters
$\gamma_{10}$, $\gamma_{01}$, $\delta$ and $\alpha$, on each level.  
For this example the plot on the left shows the scaling parameter 
$\delta \goto 1$, rapidly, indicating accuracy of the Jacobian terms on the coarse mesh. 
The numerical dissipation parameter $\gamma_{10}$ hovers at its maximum value for 
the first sequence of updates, then decays steadily to one, as the updates
converge to within tolerance given in \eqref{eqn:cor_decr_002} of 
Corollary \ref{cor:fast_gamma_decr}. 
Decrease of $\gamma_{10}$ as in Lemma \ref{lemma:gen_gmon} is not relevant as  
for $\GM = 5$ and $q = 0.865$ in this example, $\gmo$ given by \eqref{eqn:gmon_def}
yields $\gmo < 1$.  The Picard-like 
regularization $\gamma_{01}$, shows a few spikes above it's baseline level close
to $\gamma_{10}$, and indeed remains active into the asymptotic regime, showing that
the additional numerical diffusion maintains some cancellation properties against the
linearization error.  It was also observed numerically if this parameter were suppressed,
that is $\sigma_{01} = 0$ meaning $\gamma_{01} = \gamma_{10}$, the iterates tended to
diverge after level 30.

The plot on the right side of Figure~\ref{fig:ex1Rparam} shows the terminal value
of the Tikhonov-like parameter $\alpha$ together with the norm of the scaled
regularization term $\alpha\nr{R\w^n}$, and terminal residual $\nr{r_k}$
on each refinement level $k$.  Here it is seen that due to the scaling of $\alpha$ by
$\gamma_{10}/\nr{R\w^n}$, the parameter $\alpha$ does not go to zero, however
the definition \eqref{eqn:update_alpha} keeps the level of contributed regularization 
below the norm of the residual. Indeed, the plot of $\alpha$ diverges from
the plot of $\alpha \nr{R\w^n}$ as $\gamma\goto 1$ and $\nr{R\w^n}$ decreases
into the asymptotic regime.
It is remarked however, that the initial $\alpha_k^0 = \nr{r_k^0}$ on each refinement, 
then $\alpha$ decreases with the residual over each iteration. 
Without maintaining this low level of regularization into the asymptotic regime, 
the iterations were again observed to diverge.

The second example illustrates the updated regularization parameters on the model
problem shown in previous work by the author \cite{Pollock14a,Pollock15a,Pollock15b}.
In contrast to Example \ref{example1}, here the parameters $\gamma_{10}$ and $\delta$
play a dominant role in the regularization, while $\sigma_{01}$ and $\alpha$ are 
less significant into the asymptotic regime.  

\begin{example}[Diffusion with thin layers]\label{example2} Consider the 
quasilinear diffusion equation on $\Omega = (0,1) \times (0,1)$
\begin{align}
-\div (\kappa(u) \grad u) = f(x,y) \inn \Omega, ~ u = 0 ~\on~ \pa \Omega,
\end{align}
with 
\begin{align}
\kappa(u) &=  k + \f {1}{\eps + (u-a)^2},  
\end{align}
with the parameters $a = 1/2$, $\eps = 10^{-5}$,
and $k = 1$. 
The source function $f$ is chosen so the exact solution 
$u(x,y) = \sin(\pi x)\sin(\pi y)$. This problem, featuring a bounded second derivative
and high regularity of the solution, fits into both problem classes mentioned in
Remark~\ref{remark:known_classes}.

The initial regularization parameter $\gamma_{10}$ is set as
$\GM = (\sqrt 3/2)\eps^{-1/2}$, the approximate ratio of $\kappa'(\bar s)/\kappa(\bar s)$,
where $\bar s = \argmax (\kappa'(s)) $.
The regularization function $\phi(w,v) = (\grad w, \grad v)$, the standard
Laplacian preconditioner.
\end{example}
\begin{figure}
\centering
\includegraphics[trim=120pt 80pt 100pt 110pt,
clip=true,width=0.30\textwidth]{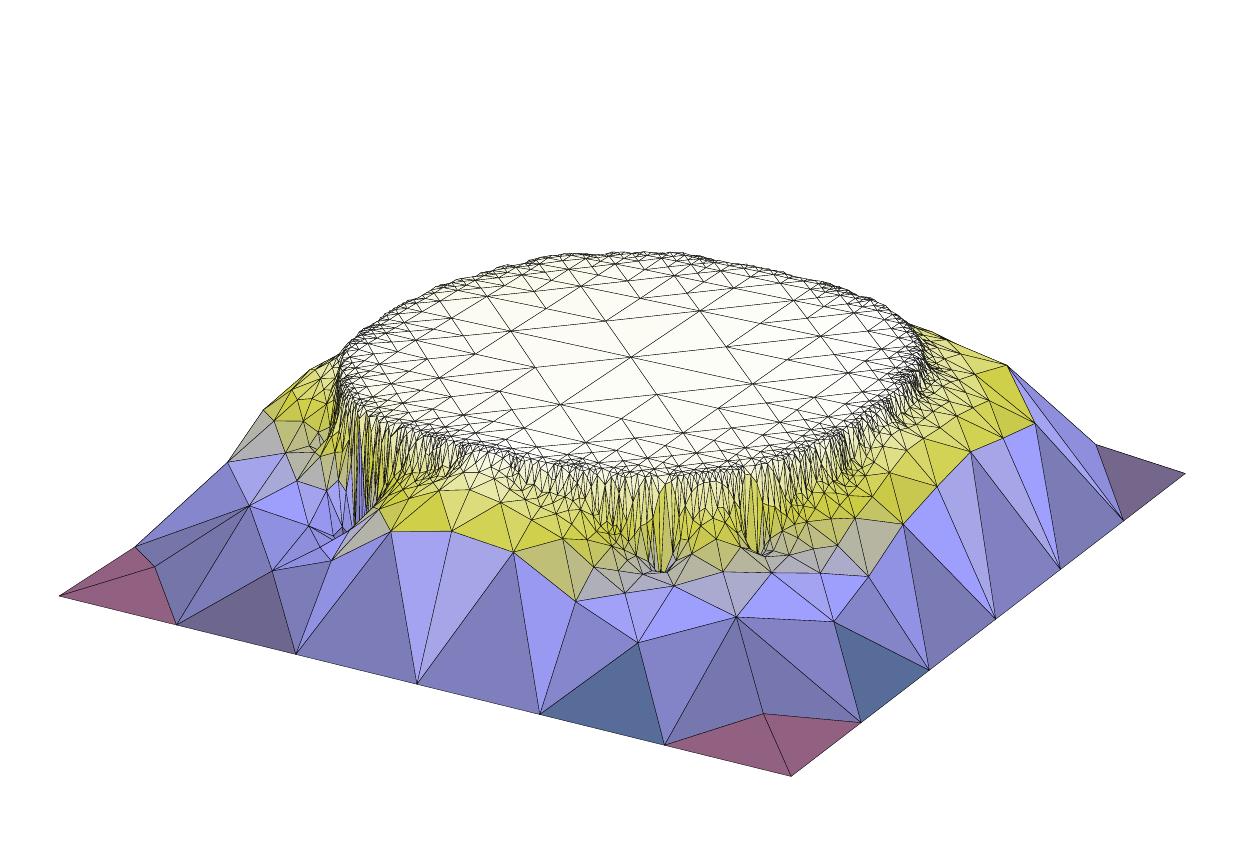}~\hfil~
\includegraphics[trim=120pt 80pt 100pt 110pt,
clip=true,width=0.30\textwidth]{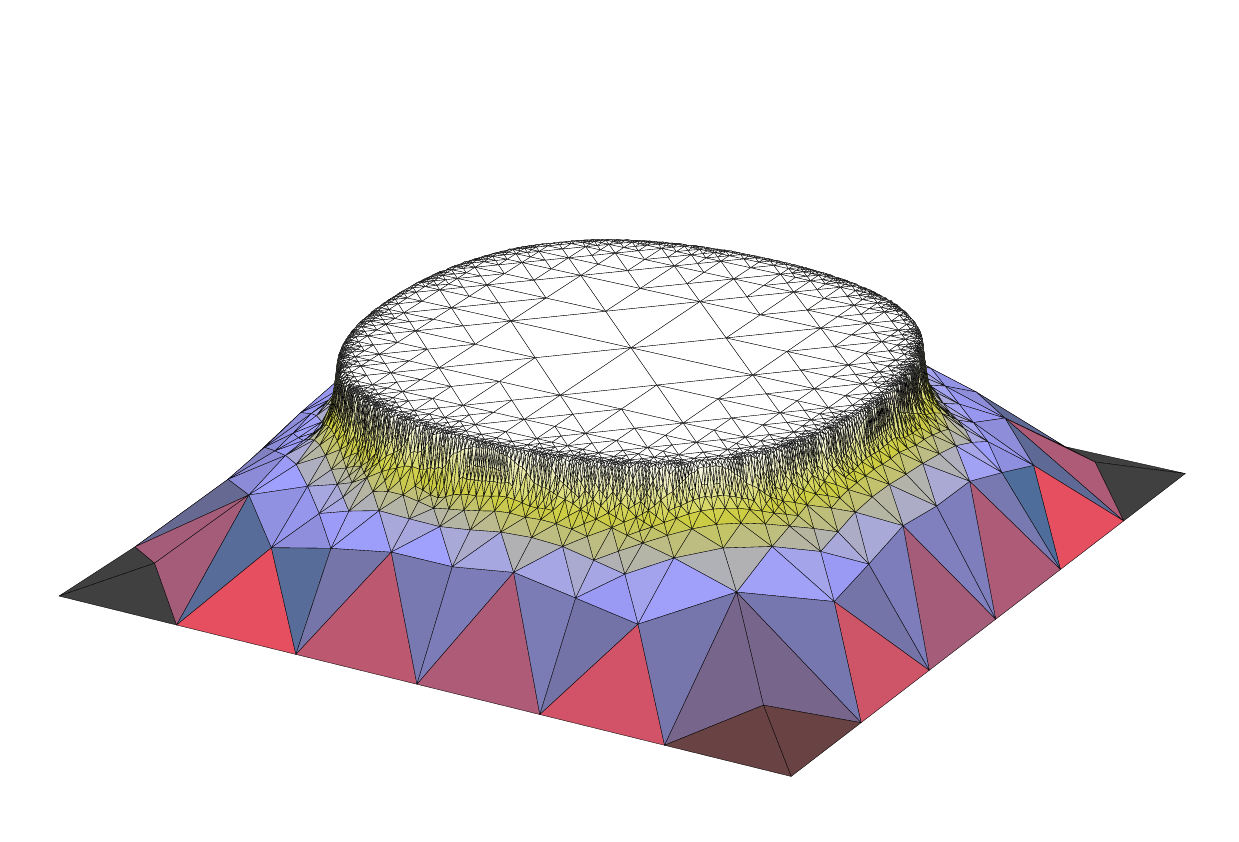}~\hfil~
\includegraphics[trim=120pt 80pt 100pt 110pt,
clip=true,width=0.30\textwidth]{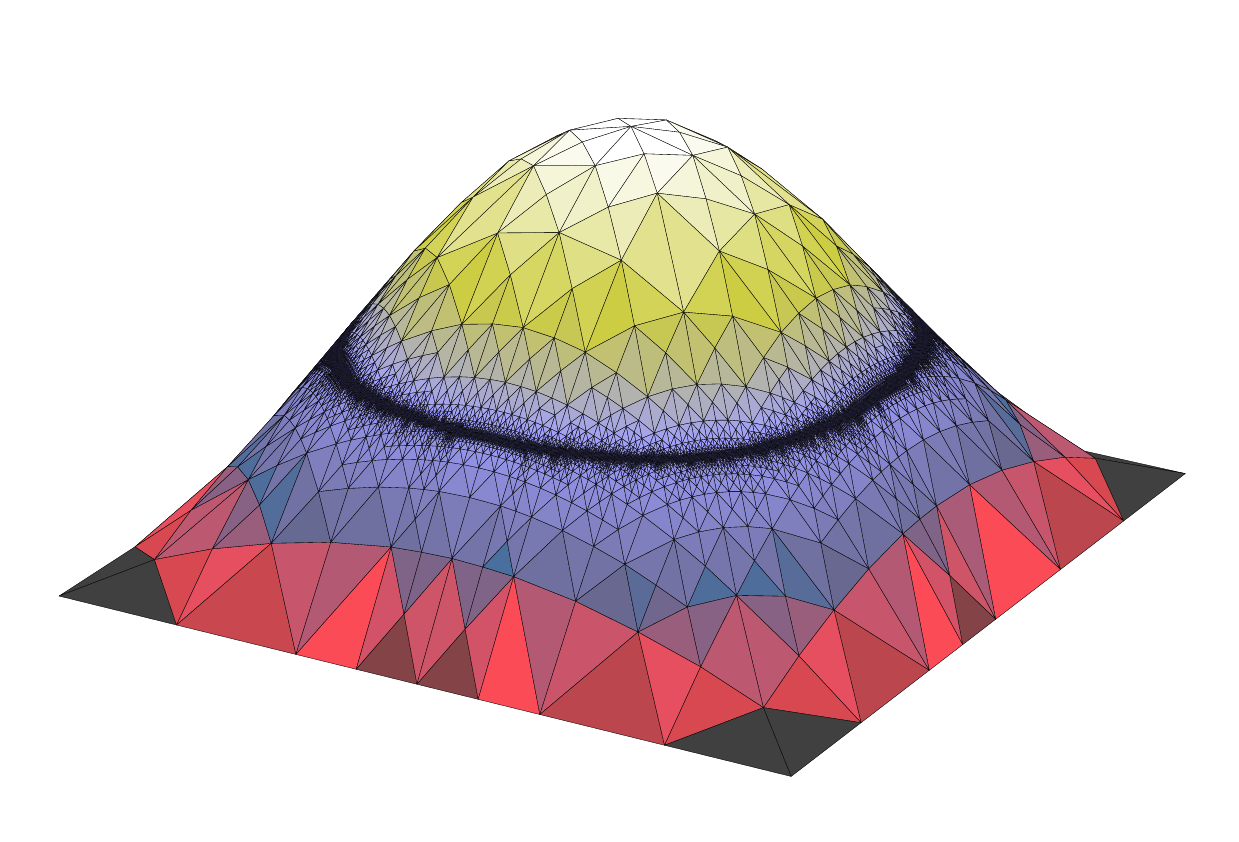}~\hfil~
\caption{Terminal solution iterates from Example~\ref{example2}.
Left: solution iterate with $\gamma_{10}= 11$ on level 25 with 1511 dof.  
Center: solution iterate with $\gamma_{10}=5$ on level 30 with 3062 dof.
Right: solution iterate with $\gamma_{10} = 1$ on level 40 with 21678 dof. 
}
\label{fig:ex2sol}
\end{figure}
\begin{figure}
\centering
\includegraphics[trim=40pt 40pt 40pt 40pt,
clip=true,width=0.32\textwidth]{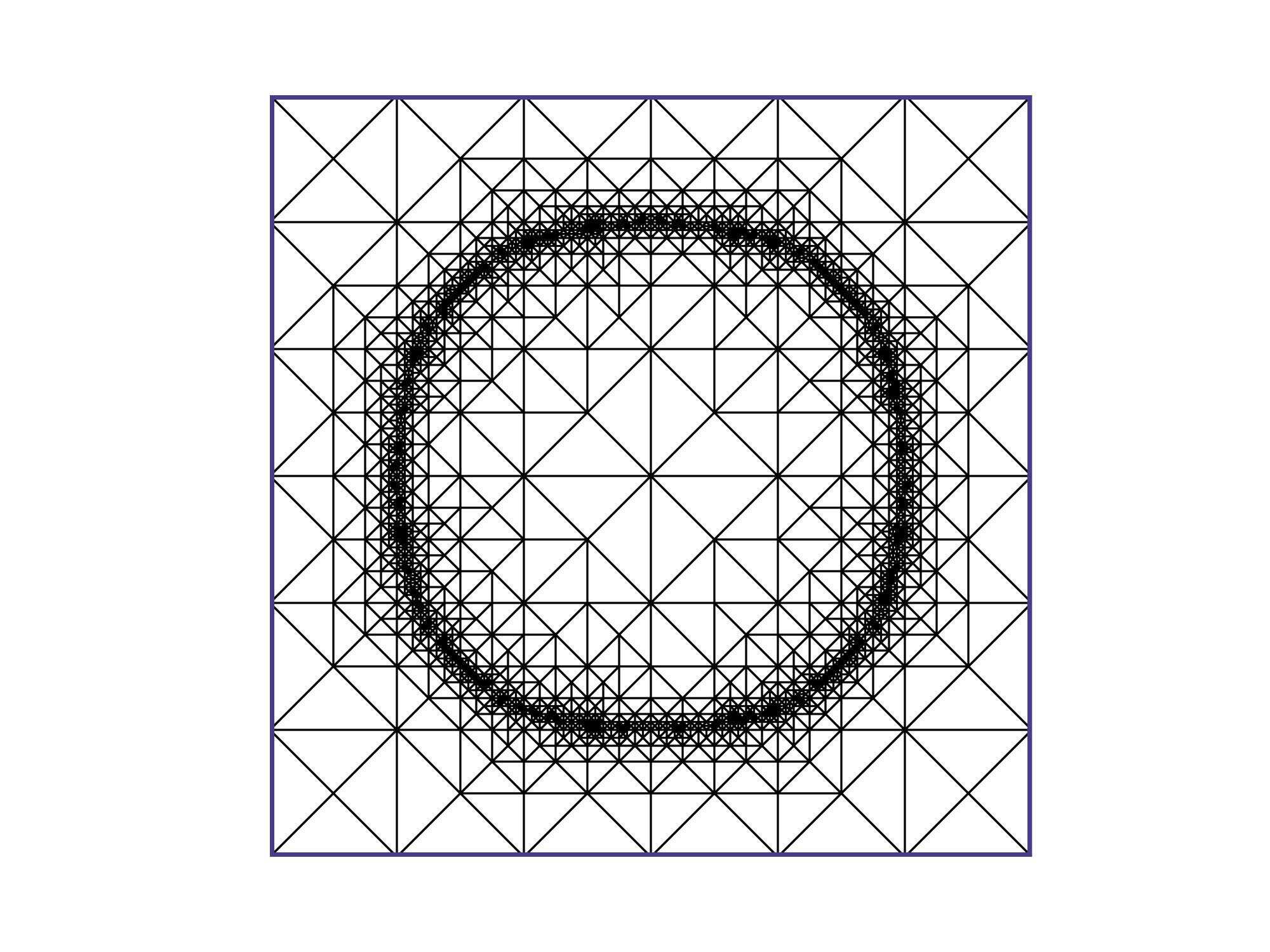}~\hfil~
\includegraphics[trim=40pt 40pt 40pt 40pt,
clip=true,width=0.32\textwidth]{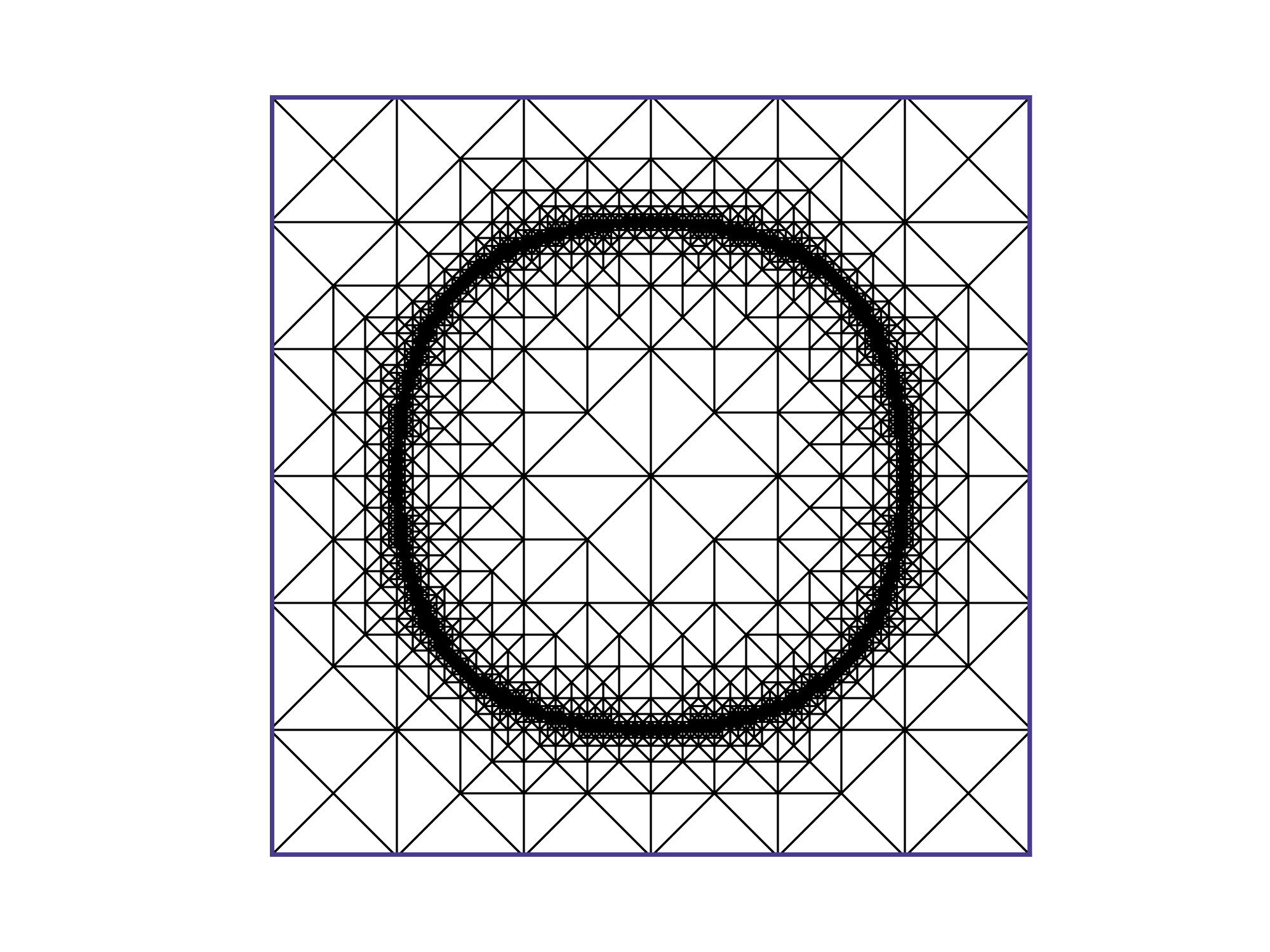}~\hfil~
\includegraphics[trim=40pt 40pt 40pt 40pt,
clip=true,width=0.32\textwidth]{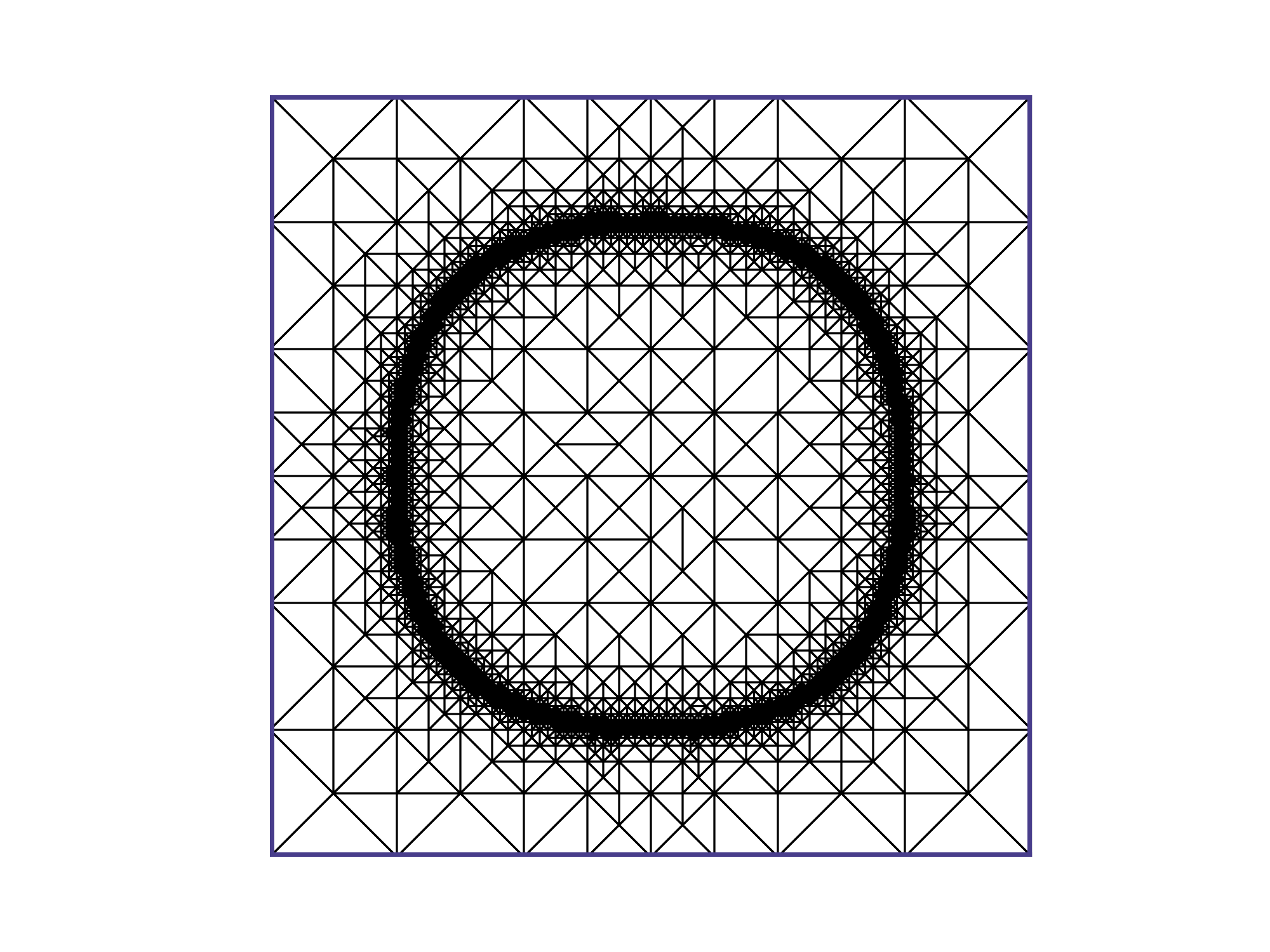}~\hfil~
\caption{Adpative meshes from Example~\ref{example2}.
Left: mesh on adaptive level 25 with 1511 dof.  
Center: mesh on adaptive level 30 with 3062 dof.
Right: mesh on adaptive level 40 with 21768 dof. 
}
\label{fig:ex2mesh}
\end{figure}

Figure \ref{fig:ex2sol} shows snapshots of the solution progression through the 
preasymptotic and into the asymptotic regime.  The snapshot on the left, from
level 25 with 1511 dof, and the snapshot in the center, from level 30, with 3062 dof,
show the effect of $\delta \ll 1$: the source function is scaled down so the solution 
iterates are held beneath the strong diffusion layer at $u = 1/2$, until the 
diffusion in the vicinity of the ultimately thin layer is sufficiently resolved. 
Then, as $\delta \goto 1$,
the full strength source pushes the solution iterates through the diffusion layer,
resulting in the asymptotic iterate on the right of Figure \ref{fig:ex2sol}, on 
level 40 with 21678 dof. The corresponding meshes in Figure \ref{fig:ex2mesh} 
illustrate the mesh refinement focused in the vicinity of the steep gradients
of the diffusion layer.

\begin{figure}
\centering
\includegraphics[trim=5pt 0pt 5pt 0pt, clip=true,width=0.32\textwidth]
{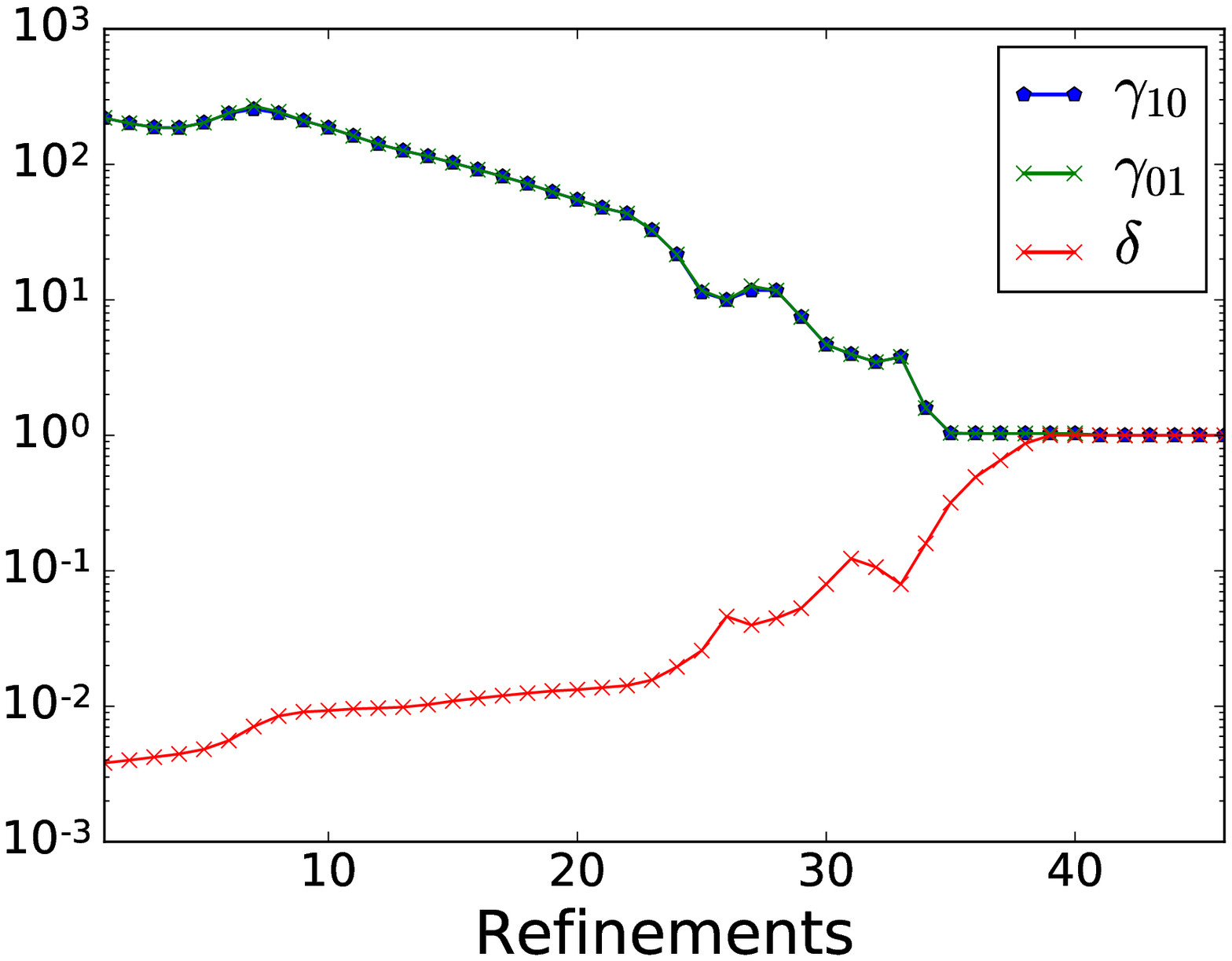}
\includegraphics[trim=5pt 0pt 5pt 0pt, clip=true,width=0.32\textwidth ]
{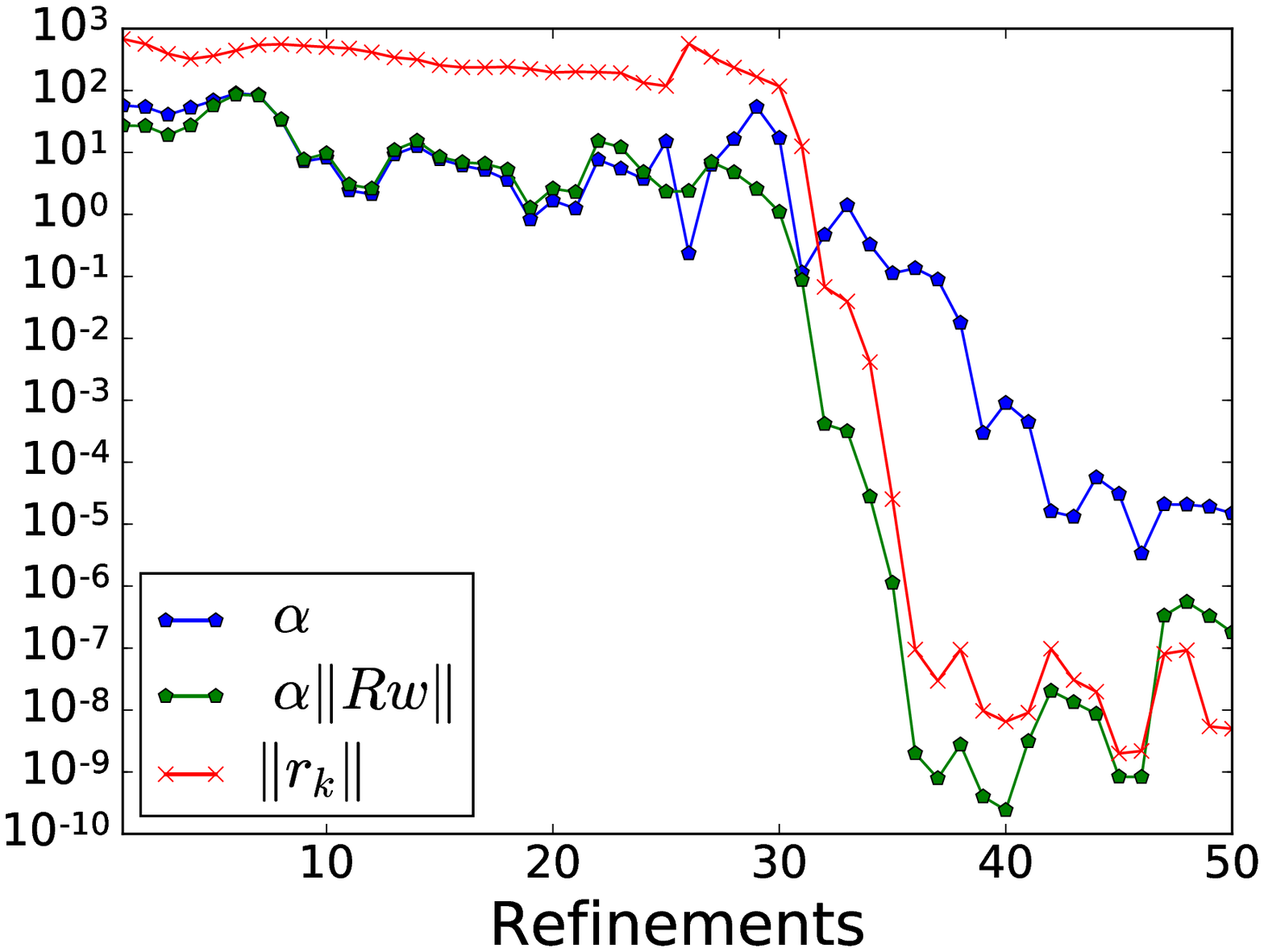}~
\includegraphics[trim=5pt 0pt 5pt 0pt, clip=true,width=0.32\textwidth ]
{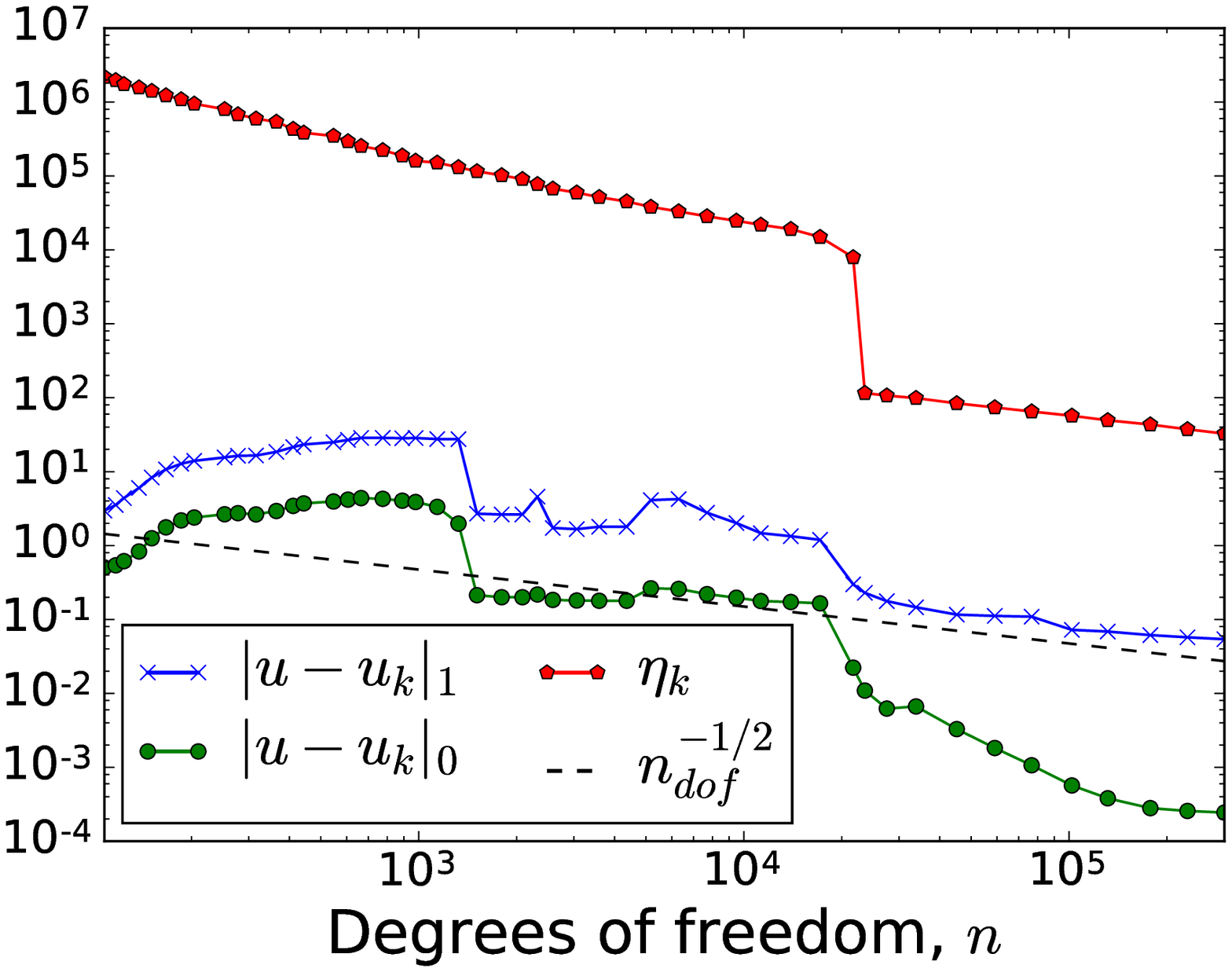}~
\caption{Left: regularization parameters $\gamma_{10}$, $\gamma_{01}$  and $\delta$.
       Center: regularization parameter $\alpha$, the norm of the regularization
               $\alpha\nr{R\w}$, and the norm of the terminal residual
               $\nr{r_k}$.
        Right: $H^1$ error, $|u - u_k|_1$; $L_2$ error, $|u-u_k|_0$; and,
               error estimator $\eta_k$, against $n_{dof}^{-1/2}$, degrees
               of freedom, for Example~\ref{example2}, with nonlinear iterations running 
               to tolerance $\tol = 10^{-7}$.}
\label{fig:ex2Rparam}
\end{figure}

The first two plots of Figure \ref{fig:ex2Rparam} shows the terminal values of each 
regularization parameter
$\gamma_{10},\gamma_{01},$ $\delta$ and $\alpha$ on each refinement level.
On the left, it is seen as $\gamma_{10}$ progresses from $\GM$ of approximately
250 down to 1, the initial relationship $\delta = 1/\gamma_{10}$ is roughly 
maintained.  From this plot it is also apparent that the regularization parameter
$\sigma_{01}$ plays only a minor role in the stabilization of the Jacobian, and
$\gamma_{01} \approx \gamma_{10}$ throughout the simulation.  

Figure~\ref{fig:ex2Rparam} in the center, shows the terminal value of $\alpha$, 
which scales the Tikhonov-like regularization term, plotted together with the
full norm of the Tikhonov-like term $\alpha\nr{R\w^n}$, and the 
final residual on each iteration $\nr{r_{k}}$.  The effect of scaling $\alpha$ 
against the norm of $R\w^n$ is seen to be small in the preasymptotic regime 
where $\nr{R\w^n} = \bigo(1)$. This scaling is however of increasing importance 
into the asymptotic phase, to reduce this
regularization to the order of the residual norm, for fast convergence.

After refinement level 28, as $\gamma_{10} < \gmo = 96$, given by 
Lemma \ref{lemma:gen_gmon}, the 
residual decrease exit criteria, \eqref{exit:pa1}-\eqref{exit:pa2} of Condition (2), 
Critia \ref{criteria:exit}, is enforced, resulting in the rapid decrease 
of the residual over the next several refinements, seen in 
Figure \ref{fig:ex2Rparam} on the right. Finally, it is noted 
in the plot on the right of Figure \ref{fig:ex2Rparam}, that upon entering
the asymptotic regime with the residual solving to tolerance at each iteration,
the $H^1$ error reduces at the rate $n_{dof}^{-1/2}$, the expected rate 
for the corresponding linear problem.

\section{Conclusion}\label{sec:conclusion}
This paper
describes a framework for pseudo-time regularization, applied to a generally
nonmonotone class of quasilinear partial differential equations.
The regularization, which is designed to exploit the quasilinear structure of
the equation, is first derived from the discrete problem at the PDE level. 
The regularized linear algebraic system is then specified under inexact assembly.
The residual representation of the assembled system then reveals the errors 
induced from regularization, linearization and floating-point arithmetic;
and, allows insight into how regularization can control the linearization error.
An updated set of regularization parameters is presented, then applied 
to an adaptive algorithm to approximate the solution of quasilinear PDE of
nonmonotone type. The method is demonstrated on two problems, the first of 
which features an anisotropic diffusion coefficient that is not twice differentiable.
The second demonstrates recovering a known solution from a model problem with 
a thin diffusion layer.
The results suggest theoretical convergence of the error without a sufficiently-fine
mesh condition or a second derivative on the solution-dependent diffusion 
coefficient should be possible. 